\documentclass[a4paper,10pt]{amsart}  
\usepackage[latin1]{inputenc} 
\usepackage{amssymb}
\usepackage{colordvi}
\usepackage{graphicx,tikz}
\usetikzlibrary{patterns}
\usepackage{vmargin}
\usepackage{todonotes}
\usepackage{tikz}
\usetikzlibrary{arrows}
\setmargrb{1in}{1in}{1in}{1in} 
\newtheorem{theorem}{Theorem}
\newtheorem{proposition}[theorem]{Proposition}
\newtheorem{lemma}[theorem]{Lemma}
\newtheorem{corollary}[theorem]{Corollary}

\theoremstyle{definition}
\newtheorem{example}[theorem]{Example}

\newtheorem*{rep@theorem}{\rep@title}
\newcommand{\newreptheorem}[2]{%
\newenvironment{rep#1}[1]{%
 \def\rep@title{#2 \ref{##1}}%
 \begin{rep@theorem}}%
 {\end{rep@theorem}}}
\makeatother

\newreptheorem{theorem}{Theorem}

\definecolor{darkblue}{rgb}{0,0,0.7} 
\newcommand{\darkblue}{\color{darkblue}} 
\newcommand{\Dfn}[1]{\emph{\darkblue #1}} 

\DeclareMathOperator{\Ref}{ref}

\DeclareMathOperator{\ran}{ran}
\DeclareMathOperator{\Span}{span}
\DeclareMathOperator{\ideal}{ideal}

\title{The Zonotopal Algebra of the Broken Wheel Graph and its Generalization}

\author[S.B.~Brodsky]{Sarah~B.~Brodsky$^{\otimes}$}
\address[S.B.~Brodsky]{Department of Mathematics, Technische Universität Berlin, 10623 Berlin, Deutschland}
\email{brodsky@math.tu-berlin.de}
\thanks{$^{\otimes}$With the support of the European Research Council grant SHPEF awarded to
Olga Holtz and the Berlin Mathematical School}

\begin{document}

\maketitle

\begin{abstract}
The machinery of zonotopal algebra is linked with two particular polytopes: the Stanley-Pitman polytope and the regular simplex $\mathfrak{Sim}_n(t_1,...,t_n)$ with parameters $t_1,...,t_n\in \mathbb{R}_+^n$, defined by the inequalities $\sum_{i=1}^n r_i\leq \sum_{i=1}^n t_i, \mbox{ } r_i\in \mathbb{R}_+^n,$ where the $(r_i)_{i\in [n]}$ are variables. 

Specifically, we will discuss the central Dahmen-Micchelli space of the broken wheel graph $BW_n$ and its dual, the $\mathcal{P}$-central space. We will observe that the $\mathcal{P}$-central space of $BW_n$ is monomial, with a basis given by the $BW_n$-parking functions. We will show that the volume polynomial of the the Stanley-Pitman polytope lies in the central Dahmen-Micchelli space of $BW_n$ and is precisely the polynomial in a particular basis of the central Dahmen-Micchelli space which corresponds to the monomial $t_1t_2\cdots t_n$ in the dual monomial basis of the $\mathcal{P}$-central space. 

We will then define the generalized broken wheel graph $GBW_n(T)$ for a given rooted tree $T$ on $n$ vertices. For every such tree, we can construct $2^{n-1}$ directed graphs, which we will refer to as \textit{generalized broken wheel graphs}. Each generalized broken wheel graph constructed from $T$ will give us a polytope, its volume polynomial, and a \textit{reference monomial}. The $2^{n-1}$ polytopes together give a polyhedral subdivision of $\mathfrak{Sim}_n(t_1,...,t_n)$, their volume polynomials together give a basis for the subspace of homogeneous polynomials of degree $n$ of the corresponding central Dahmen-Micchelli space, and their reference monomials together give a basis for its dual. 
\end{abstract}

\tableofcontents

\section{Introduction}

The theory of zonotopal algebras introduced by Holtz and Ron \cite{HR} gives a means of associating some of the most fundamental objects in combinatorics to solution sets of differential equations. Starting with a \textit{box-spline}, the central Dahmen-Micchelli space can be constructed: a space of polynomials which satisfies the same differential equations as the polynomials locally describing the starting box-spline. The central Dahmen-Micchelli space is the Macaulay inverse system of an ideal generated by powers of linear forms; these linear forms are indexed by the cocircuits of the matroid whose ground set consists of the vectors defining the underlying zonotope of the starting box-spline. Holtz and Ron \cite{HR} also define a dual space to the central Dahmen-Micchelli space, the $\mathcal{P}$-central space, which has the same Hilbert polynomial as the central Dahmen-Micchelli space and can be associated to a hyperplane arrangement derived from a power ideal in which the $\mathcal{P}$-central space is the Macaulay inverse system of. There is also the internal and external Dahmen-Micchelli spaces and their duals as well, leaving us with many algebraic objects to play with.\

Having this strong bridge between approximation theory (via the box-spline) and combinatorics is powerful. But the question still remains, where can this powerful bridge be applied? Here we link the machinery of zonotopal algebra with two particular polytopes, showing that the zonotopal spaces derived from two particular graphs captures the volumes of these polytopes, as well as the volumes of polytopes appearing in particular polyhedral subdivisions of these polytopes.

The first of the two is the Stanley-Pitman polytope. The Stanley-Pitman polytope, introduced by Stanley and Pitman \cite{PS}, has a polyhedral subdivision whose chambers are indexed naturally by rooted binary trees, giving us a representation of the associahedra. For $t\in \mathbb{R}_+^n$, the Stanley-Pitman polytope is specifically the $n$-dimensional polytope $Q_n(t)$ defined by the equations $$Q_n(t):=\{r\in \mathbb{R}_+^n : \sum_{i=j}^n r_i\leq \sum_{i=j}^n t_i, 1\leq j\leq n\},$$ where we define $\mathbb{R}_+ := [0,\infty)$. Stanley and Pitman study the volume of $Q_n(t)$, $$q_n(t):=\operatorname{vol}(Q_n(t)),$$ and show in \cite{PS} that $q_n(t)$ is a polynomial which is the sum of exactly $C_n :=\frac{{2n \choose n}}{n+1}$ normalized monomials.

\begin{proposition}[Pitman and Stanley, \cite{PS}]\label{r1}
For each $n\in \mathbb{N}\backslash \{0\}$, we have that $$q_n(t)=\sum_{k\in K_n} \prod_{i=1}^n \frac{t_i^{k_i}}{k_i!}=\frac{1}{n!}\sum_{k\in K_n} {n \choose k_1,...,k_n}t_1^{k_1}\cdots t_n^{k_n},$$ where $$K_n :=\{k\in \mathbb{N}^n : \sum_{i=1}^j k_i\geq j\mbox{ for all } 1\leq i\leq n-1\mbox{ and } \sum_{i=1}^n k_i=n\}$$ with $\mathbb{N}:=\{0,1,2,...\}$.
\end{proposition}

The volume $q_n(t)$ of the Stanley-Pitman polytope $Q_n(t)$ can be captured via the zonotopal algebra of the \textit{broken wheel graph} $BW_n$: a finite undirected graph with $n + 1$ vertices and $2n$ edges, which defines the graphical matroid needed for our constructions. In section \ref{BWG}, we will rigorously define the broken wheel graph $BW_n$, define what it means to be a \textit{parking function} of $BW_n$, and discuss some properties of such parking function. We will then use these properties in section \ref{zonotopeBWG}, where we will discuss the Tutte polynomial and Hilbert series of $BW_n$, as well as develop the zonotopal algebra of $BW_n$, after giving a review of the general theory of zonotopal algebra. Section \ref{SPpolytope} of this paper will specifically address the Stanley-Pitman polytope and use the machinery developed to prove that the Stanley-Pitman volume polynomial $q_n(t)$ is the monic polynomial in the central Dahmen-Micchelli space of $BW_n$ which corresponds to the parking function $(1,...,1)\in \mathbb{R}^n$, and that it is the unique internally monic polynomial of maximal degree in the internal Dahmen-Micchelli space of $BW_n$ which corresponds to the unique internal parking function $(1,1,...,1,0)\in \mathbb{R}^{n+1}$. Using the following notation, we will also further characterize the volume polynomial $q_n(t)$ with the following two theorems: denote partial differentiation with respect to $t_i$ by $D_i$; i.e with $p_i : \mathbb{R}^n\rightarrow \mathbb{R}^n, t\mapsto t_i,$  we have $D_i:=p_i(D)$, and $D_0 := 0$. We then have that:

\begin{theorem}\label{t1}
The polynomial $q_n(t)$ is the only polynomial (up to normalization) of degree $n$ that is annihilated by each of the following differential operators $$D_i(D_i-D_{i-1}), \mbox{ } i=1,...,n.$$ Moreover, let $\mathcal{P}_{n,j}$ be the subspace of homogeneous polynomials (in $n$-indeterminates) of degree $j$ that are annihilated by each of the above differential operators. Then:
\begin{enumerate}
\item $\mathcal{P}_{n,j}$ lies in the span of the translates of $q_n(t)$.
\item dim $\mathcal{P}_{n,j}$ = ${n \choose j}$.
\end{enumerate}
\end{theorem}

\begin{theorem}\label{t2}
The polynomial $q_n(t)$ is the only polynomial $q(t)$ (in $n$ variables) that satisfies the following two properties:
\begin{enumerate}
\item With $m$ the square-free monomial $$m: t\mapsto \prod_{i=1}^{n} t_i,$$ the monomial support of $(q-m)(t)$ is disjoint of the monomial support of the polynomial $$t\mapsto \prod_{i=1}^n (t_i+t_{i-1}), \mbox{ } t_0 := 0.$$
\item $q(t)$ is annihilated by each of the following differential operators: $$(D_{j+1}-D_j)(\prod_{k=i}^j D_k)(D_i-D_{i-1}), \mbox{ } 1\leq i \leq j < n,$$ and $$ (\prod_{k=1}^n D_k)(D_i-D_{i-1}),\mbox{ } 1\leq i\leq n.$$
\end{enumerate}
\end{theorem}

We will then review the polyhedral subdivision of $Q_n(t)$ given by Pitman and Stanley \cite{PS}, whose set of interior faces, ordered by inclusion, is isomorphic to the face lattice of the dual associahedron, and note how the volume of each polytope in  this subdivision is captured by the zonotopal algebra of the broken wheel graph. This observation is motivation for studying the volumes of polyhedral subdivisions in terms of zonotopal algebras and lead us to our study of the second polytope.

In section \ref{GBWG} we will introduce the second polytope: the regular simplex $\mathfrak{Sim}_n(t_1,...,t_n)$ with parameters $t_1,...,t_n\in \mathbb{R}_+^n$, defined by the inequalities $$\sum_{i=1}^n r_i\leq \sum_{i=1}^n t_i, \mbox{ } r_i\in \mathbb{R}_+^n,$$ where the $(r_i)_{i\in [n]}$ are variables. For every rooted tree $T$ with $n$ vertices, we can construct $2^{n-1}$ directed graphs, which we will refer to as \textit{generalized broken wheel graphs}. Each generalized broken wheel graph constructed from $T$ will give us a polytope, its volume polynomial, and a \textit{reference monomial}. The $2^{n-1}$ polytopes together give a polyhedral subdivision of $\mathfrak{Sim}_n(t_1,...,t_n)$, their volume polynomials together give a basis for the subspace of homogeneous polynomials of degree $n$ of the corresponding central Dahmen-Micchelli space, and their reference monomials together give a basis for its dual. And so, for each rooted tree with $n$ vertices we have a polyhedral subdivision of $\mathfrak{Sim}_n(t_1,...,t_n)$ completely characterized by the zonotopal algebra of the generalized broken wheel graphs constructed from $T$.

Our study provides intriguing and quite rich examples of zonotopal algebra, on the one hand, and sheds new light on how volumes of polytopes, and their polyhedral subdivisions, can be studied on the other. This paper is meant for both the eyes of those familiar and unfamiliar with the study of zonotopal algebras. For those familiar, we hope to provide you with an enriching application which will spark your further interest. For those unfamiliar, we hope to illustrate to you the potential of zonotopal algebras as a combinatorial way to connect to analytic tools.

\section{The Broken Wheel Graph}\label{BWG}

Before we jump into the details of the broken wheel graph, let's define it precisely. The \Dfn{broken wheel graph} $BW_n$ is a finite undirected graph with $n+1$ vertices $[0:n]$ and $2n$ edges. The root vertex $0$ is connected twice to the vertex $1$, and once to each other vertex. In addition, a single edge connects each consecutive pair $i$ and $i+1$, with $i=1,...,n-1$.

A \Dfn{wheel graph} $W_n$ consists of the edges of a regular $n$-gon, together with all the radii that connect the vertices of the $n$-gon to its center. In algebraic graph theory, the $n$ verities of the $n$-gon are associated with the standard basis $(e_i)_{i=1}^n$ of $\mathbb{R}^n$, while the center is identified with $e_0:=0$. The edge that connects the vertices $i$ and $j$ is realized by the vector $e_i-e_j$ (or $e_j-e_i$, as the sign will not matter for us). For certain purposes (such as the definition of the internal activity and the external activity of the forests of the graph) it is necessary to order the edges, viz. their vector realization. The order that serves our needs is as follows: $$x_{2i-1}=e_i-e_{i-1}, x_{2i}=e_i, i=1,2,...,n.$$ The vectors $X_n:=(x_1,...,x_{2n})$ correspond to the edges of the wheel $W_n$: odd numbered vectors corresponding  to the edges of the $n$-gon and the even vectors corresponding to the radii. Note that we have written $x_1=e_1-e_0=e_1$. This is because the broken wheel $BW_n$ is obtained from the wheel $W_n$ when replacing the $n$-gon edge $e_1-e_n$ by the radius $e_1$. Thus, the edge $e_1$ is doubled in $BW_n$.\

Let $(e_i)_{i=1}^n$ denote the standard basis for $\mathbb{R}^n$ and let $e_0:=0$. Let's identify each vertex $0\leq i\leq n$ of $BW_n$ with the vector $e_i$ and each edge that connects vertex $i$ to vertex $j>i$ with the vector $e_j-e_i$. Letting $$x_{2i-1}=e_i-e_{i-1},\mbox{ } x_{2i}=e_i,\mbox{ } i=1,...,n,$$ we then use the following order on the edge set of $BW_n$: $$BW_n:=(x_1\prec x_2\prec \cdots \prec x_{2n}).$$

\begin{figure}[htbp]
\begin{center}
 \begin{tikzpicture}
\tikzset{vertex/.style = {shape=circle,draw,minimum size=1.5em}}
\tikzset{edge/.style = {-,> = latex'}}
\node[vertex] (a) at  (0,0) {0};
\node[vertex] (b) at  (0,2) {1};
\node[vertex] (c) at  (-2,0) {2};
\node[vertex] (d) at  (0,-2) {3};
\node[vertex] (e) at  (2,0) {4};
\draw[edge]  (a) edge node [above] {$x_4$} (c);
\draw[edge] (a) edge node [left] {$x_6$} (d);
\draw[edge] (a) edge node [above] {$x_8$} (e);

\draw[edge] (a)  edge[bend left] node [left] {$x_2$} (b);
\draw[edge] (a)  edge[bend right] node [right] {$x_1$} (b);
\draw[edge] (b)  edge[bend right] node [left] {$x_3$} (c);
\draw[edge] (c)  edge[bend right] node [left] {$x_5$} (d);
\draw[edge] (d)  edge[bend right] node [right] {$x_7$} (e);
\end{tikzpicture}
\caption{The broken wheel graph $BW_4$.}
\label{BW4}
\end{center}
\end{figure}
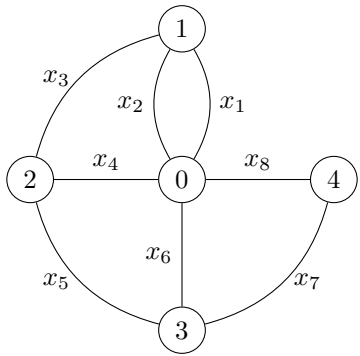

With this order, the edges of $BW_n$ form the columns of an $n\times 2n$ matrix denoted $X_n$. For example, the matrix $X_4$ is $$X_4=\left[\begin{array}{cccccccc}1 & 1 & -1 & 0 & 0 & 0 & 0 & 0 \\0 & 0 & 1 & 1 & -1 & 0 & 0 & 0 \\0 & 0 & 0 & 0 & 1 & 1 & -1 & 0 \\0 & 0 & 0 & 0 & 0 & 0 & 1 & 1\end{array}\right].$$\vspace{.01in}

With this identification, ordering of the edges of $BW_n$, and the matrix $X_n$ we have enough to construct three pairs of polynomial spaces, which are examples of the fundamental pairs of polynomial spaces studied generally in the field zonotopal algebra. Before we do this (in section \ref{zonotopeBWG}), we need to talk about the \textit{parking functions} of $BW_n$, as they are key to discussing these pairs of spaces.\

\subsection{The Parking Functions of the Broken Wheel Graph}

Given a subset of vertices $[i:j]$ of $BW_n$ and a vertex $k\in [i:j]$, we denote by $$d(i,k,j)$$ the \Dfn{out-degree} of $k$, viz. the number of edges that connect $k$ to vertices in the complement of $[i:j]$. Note that $d(i,k,j)\in \{1,2,3\}$, $0<i\leq k\leq j\leq n$, for $BW_n$. Parking functions of graphs are studied in generality by Postnikov and Shapiro in \cite{2003math......1110P}.  Following their definition, a \Dfn{parking function} of the graph is a function $s\in \mathbb{N}^n$, with $s(i)$ denoting the $i^{th}$ entry of $s$, which satisfies the following condition: given any $1\leq i\leq j\leq n$, there exists a $k\in [i:j]$ such that $s(k)<d(i,k,j)$. This definition follows suit from the definition of parking functions given in \cite{PS}. A parking function $s$ is called an \Dfn{internal parking function} of a graph if for every $1\leq i\leq j\leq n$, we either have a $k\in [i:j-1]$ such that $s(k)<d(i,k,j)$ or $s(j)<d(i,j,j)-1$. Let the set of parking functions of $BW_n$ be denoted by $$S(BW_n)$$ and the set of internal parking functions of $BW_n$ by $$S_-(BW_n).$$ 

\begin{lemma}\label{l6}
If $s$ is a parking function of $BW_n$, then $\prod_{k=i}^j s(k)\leq 2$, while $\prod_{k=i}^n s(k)\leq 1$, for every $1\leq i\leq j\leq n$.
\end{lemma}

\begin{proof}
Now let's consider $s\in S(BW_n)$. If $i=j=n$, then the only $k$ we can choose is $k=n$ and we must then have that $s(n)\leq 1$, as $d(n,n,n)=2$ for $BW_n$. If we choose $i=j<n$, then the only $k$ we can choose is $k=i=j<n$ and we must then have that $s(n)\leq 2$.\

If we have that $s(i)=2$, and we choose $j$ to be $n$, then we can see that $d(i,i,n)=2$ and, for $k>i$, $d(i,k,n)=1$. We can then conclude from these two observations that $s(k)=0$ for some $k>i$, as this is the only way we can find a $k\in [i:n]$ such that $s(k)<d(i,k,n)$.\

Let's now assume a bit further that $s(i)=s(j)=2$ for some $1\leq i < j <n$. As $d(i,i,j)=d(i,j,j)=2$, while $d(i,k,j)=1$ for $i<k<j$, we can see that $s(k)=0$ for some $i<k<j$.\

From all of these observations, we can conclude that, in order for $s$ to be a parking function of $BW_n$, we must have that $\prod_{k=i}^j s(k)\leq 2$, while $\prod_{k=i}^n s(k)\leq 1$, for every $1\leq i\leq j\leq n$.
\end{proof}

Let us now define a particular subset of $S(BW_n)$ which will be necessary for our studies. The set of \Dfn{maximal parking functions} $S_{max}(BW_n)$ of $BW_n$ is defined as $$S_{max}(BW_n):=\{s\in S(BW_n) : |s|:=\sum_{i=1}^n s(i)=n\}.$$ We can explicitly define the sets $S(BW_n), S_{max}(BW_n)$, and $S_-(BW_n)$ as the \Dfn{support} of certain polynomials. For $s\in \mathbb{Z}_+^n$, let us define the monomial $$m_s: t\mapsto t^s := \prod_{i=1}^n t_i^{s(i)}.$$ Then, given a polynomial $p\in \mathbb{K}[t_1,...,t_n]$, where $\mathbb{K}$ is a field of characteristic $0$, the \Dfn{monomial support} $\mathrm{supp}\mbox{ } p(t)$ of $p(t)$ is the set of vectors $s\in \mathbb{Z}_+^n$ for which $$m_s(D)p(t) |_{t=0}\neq 0.$$

\begin{example}\label{e1}
For $q_2(t)=t_2^2/2+t_1t_2$, we have that $\mathrm{supp}\mbox{ } q_2(t) = \{(1,1),(0,2)\}$.
\end{example}

We now have the following two theorems which characterize the sets $S(BW_n), S_{max}(BW_n)$, and $S_{-}(BW_n)$ as the support of certain polynomials:

\begin{proposition}\label{p1}
For $a\in \{0,1\}$, let $$p_{n,a}(t):=\prod_{i=1}^n (a+t_{i-1}+t_i),\mbox{ } t_0:=0.$$ Then $$S_{max}(BW_n)=\mathrm{supp}\mbox{ } p_{n,0}(t) \mbox { and } S(BW_n)=\mathrm{supp}\mbox{ } p_{n,1}(t),$$ and we have that $$|S_{max}(BW_n)|=2^{n-1} \mbox { and } |S(BW_n)|\leq 2\cdot 3^{n-1}.$$
\end{proposition}

\begin{proof}
Consider the polynomial expansion of $p_{n,0}(t)$: \begin{equation}\label{pn0} p_{n,0}(t)=\prod_{i=1}^n (t_{i-1}+t_i)=(t_1^2+t_1t_2)\prod_{i=3}^n(t_{i-1}+t_i).\end{equation} We can see that $p_{n,0}(t)$ is a polynomial with $2^{n-1}$ monomials, as it is a polynomial which can be factored into $n$ binomials. Thus we have that $|\mathrm{supp}\mbox{ } p_{n,0}(t)|=2^{n-1}$.

Let us prove the equality in question by induction on $n$. First, let us assume that $n=1$. We then have that $p_{1,0}(t)=t_1,$ giving us that $\mathrm{supp}\mbox{ } p_{1,0}(t)=\{(1)\}$. Corollary \ref{c2} of this note tells us that the set of maximal parking functions is exactly the subset of $\mathbb{N}^n$ of all sequences $s$ that can be written as a sum $$s=e_1+\sum_{j=1}^{n-1} a_j,$$ with $(e_i)_{i=1}^n$ the standard basis for $\mathbb{N}^n$, and $a_j\in \{e_j,e_{j+1}\}$ for every $j$. Thus $S_{max}(BW_1)=\{(1)\}$ and we have equality for our base case.\

Now, assuming that $S_{max}(BW_k)=\mathrm{supp}\mbox{ } p_{k,0}(t)$ for $k\leq n$, we have $$p_{n+1,0}(t)=p_{n,0}(t)(t_n+t_{n+1})=p_{n,0}(t)t_n+p_{n,0}(t)t_{n+1}.$$ First, let us consider any $s\in \mathrm{supp}\mbox{ } p_{n,0}(t)t_n$. We have that the first $n-1$ entries of $s$ are going to satisfy the conditions of corollary \ref{c2}, the $n^{th}$ entry of $s$ is going to be either $1$ or $2$ (as the degree of $t_n$ for any term of $p_{n,0}(t)$ is $0$ or $1$), and that the $(n+1)^{th}$ entry of $s$ is $0$. Thus, $s$ is such a vector described in corollary \ref{c2}, meaning that $s\in S_{max}(BW_{n+1})$ and $\mathrm{supp}\mbox{ } p_{n,0}(t)t_n\subseteq S_{max}(BW_{n+1})$. Similarly, let's consider any $s\in \mathrm{supp}\mbox{ } p_{n,0}(t)t_{n+1}$. We then have that the first $n-1$ entries of $s$ satisfy the conditions of corollary \ref{c2}, the $n^{th}$ entry of $s$ is going to be either $0$ or $1$, and that the $(n+1)^{th}$ entry of $s$ is $1$. Thus, $s$ is such a vector described in corollary \ref{c2}, meaning that $s\in S_{max}(BW_{n+1})$ and $\mathrm{supp}\mbox{ } p_{n,0}(t)t_{n+1}\subseteq S_{max}(BW_{n+1})$. Thus, $\mathrm{supp}\mbox{ } p_{n+1,0}(t)=\mathrm{supp}\mbox{ } p_{n,0}(t)t_n\cup \mathrm{supp}\mbox{ } p_{n,0}(t)t_{n+1}\subseteq S_{max}(BW_{n+1})$. To show that our inclusion is actually an equality, let us assume our inclusion is strict and find a contradiction. If our inclusion is strict, then there exists an $s\in S_{max}(BW_{n+1})$ such that $s\notin \mathrm{supp}\mbox{ } p_{n+1,0}(t)$. We then have that $$m_s(D)p_{n+1,0}(t)=m_s(D)[p_{n,0}(t)(t_n+t_{n+1})]|_{t=0}=0.$$ Since we have that $m_s(D)p_{n,0}(t)|_{t=0}\neq 0$ by our induction hypothesis, we must have $s(n)=s(n+1)=0$. This means, however, that when expressing $s$ as stipulated in corollary \ref{c2}, $$s=e_1+\sum_{j=1}^{n}a_j,$$ we cannot have $a_n\in \{e_n,e_{n+1}\}$ as required. Thus, we have our contradiction and the equality desired. And so, in particular, we have that $|S_{max}(BW_n)|=2^{n-1}$.\\

Now, let us consider $p_{n,1}(t)$. We can see that $p_{n,1}(t)$ is a polynomial with $2\cdot 3^{n-1}$ terms by noting that $$p_{n,1}(t)=(1+t_1)\prod_{i=2}^{n} (1+t_{i-1}+t_i).$$

For $n=1$, we have that $p_{1,1}(t)=1+t_1$ and thus that $\mathrm{supp}\mbox{ } p_{1,1}(t)=\{(0), (1)\}$. Checking the definition of a parking function against each element of the support of $p_{1,1}(t)$, we can see that our only choice for $i$ and $j$ is $i=j=1$.  We can then see that $0<d(1,1,1)=2$ and $1<d(1,1,1)=2$; this shows us that $\mathrm{supp}\mbox{ } p_{1,1}(t)\subset S(BW_1)$. And as $|S(BW_1)|=2$, as the number of spanning trees of $SW_1$ is $2$, we have equality.\

Now, assuming $S(BW_k)=\mathrm{supp}\mbox{ } p_{k,1}(t)$ for $k\leq n$, let's consider $$p_{n+1,1}(t)=p_{n,1}(t)(1+t_n+t_{n+1})=p_{n,1}(t)+t_np_{n,1}(t)+t_{n+1}p_{n,1}(t),$$ a polynomial with at most $2\cdot 3^{n-1}$ terms; thus $|\mathrm{supp}\mbox{ } p_{n+1,1}(t)|\leq 2\cdot 3^{n-1}$. We have established in lemma \ref{l6} that if a vector $s$ is a parking function of $BW_{n+1}$ then $\prod_{k=i}^j s(k)\leq 2$ and $\prod_{k=1}^{n+1} s(k)\leq 1$ for every $1\leq i\leq j\leq n+1$. For $\mathrm{supp}\mbox{ } p_{n,1}(t)$, these conditions are met by our induction hypothesis. 

For $\mathrm{supp}\mbox{ } t_np_{n,1}(t)$, we also have that our conditions are met: $\prod_{k=1}^{n+1} s(k)=1\cdot \prod_{k=1}^{n} s(k) \leq 1$, and as $\prod_{k=i}^n s(k)\leq 1$ for every $1\leq i\leq j\leq n$, we have that $\prod_{k=i}^n s(k)\leq 2$ for every $1\leq i\leq j\leq n+1$ via our extra factor of $t_n$ in every monomial of $t_np_{n,1}(t)$. For $\mathrm{supp}\mbox{ } t_{n+1}p_{n,1}(t)$, we also have that our conditions are met, as adding a 1 to either of the products in question will not change their numerical value. As $\mathrm{supp}\mbox{ } p_{n+1,1}(t)=\mathrm{supp}\mbox{ } p_{n,1}(t)\cup \mathrm{supp}\mbox{ } t_np_{n,1}(t)\cup \mathrm{supp}\mbox{ } t_{n+1}p_{n,1}(t)$, we thus have that $\mathrm{supp}\mbox{ } p_{n+1,1}(t)\subseteq S(BW_{n+1})$.\

To prove that this inclusion is actually an equality, let us assume that the inclusion is strict and find a contradiction. If our inclusion is strict, then there exists an $s\in S(BW_{n+1})$ such that $s\notin \mathrm{supp}\mbox{ } p_{n+1,1}$. We then have that $$m_s(D)p_{n+1,1}|_{t=0}=m_s(D)[p_{n,1}(1+t_n+t_{n+1})]=0.$$ This would then imply that $m_s(D)p_{n,1}=0$, a contradiction to our induction hypothesis. Thus, we must have the equality desired. And so, in particular, we have that $|S(BW_{n+1})|\leq 2\cdot 3^{n-1}$.
\end{proof}

\begin{proposition}\label{p2}
Let $$p_{n,-}(t):=\prod_{i=1}^{n-1} (1+t_i).$$ Then $S_-(BW_n)=\mathrm{supp}\mbox{ } p_{n,-}(t)$ and $|S_-(BW_n)|=2^{n-1}$.
\end{proposition}

\begin{proof}
Let us consider the polynomial $p_{n,-}(t):=\prod_{i=1}^{n-1} (1+t_i)$ and prove our proposition by induction on $n$. As always, let's first consider our base case, $n=1$. We then have that $p_{1,-}(t)=1$. The support of this polynomial is $\mathrm{supp}\mbox{ } p_{1,-}(t)=\{(0)\}$. Following the definition of an internal parking function, as our only choice for $i$ and $j$ is $i=j=1$, we have that $0<d(1,1,1)-1=2-1=1$. Thus, we have that $(0)$ is in $S_-(BW_1)$ and that $\mathrm{supp}\mbox{ } p_{1,-}(t)=S_-(BW_1)$.\

Now, let us assume that $S_-(BW_k)=\mathrm{supp}\mbox{ } p_{k,-}(t)$ for $k\leq n$ and show that this equality is also true for $k=n+1$. We have that $p_{n+1,-}(t):=\prod_{i=1}^{n} (1+t_i)=p_{n,-}(t)(1+t_n)=p_{n,-}(t)+p_{n,-}(t)t_n$. Thus, $\mathrm{supp}\mbox{ } p_{n+1,-}(t)=\mathrm{supp}\mbox{ } p_{n,-}(t)\cup \mathrm{supp}\mbox{ } p_{n,-}(t)t_n$, where $p_{n,1}(t)$ is considered as a polynomial in $n$ variables. If $s\in \mathrm{supp}\mbox{ } p_{n,-}(t)$, then we know that for every $k\in [i:j-1]$, $1\leq i\leq j\leq n$, we have that either $s(k)<d(i,k,j)$ or $s(j)<d(i,j,j)-1$. We also know that $s(n)=0$, meaning that these inequalities certainly still hold after we extend $1\leq i\leq j\leq n$ to $1\leq i\leq j\leq n+1$. Thus, we have that $\mathrm{supp}\mbox{ } p_{n,-}(t)\subseteq S(BW_{n+1})$. If $s\in \mathrm{supp}\mbox{ } p_{n,-}(t)t_n$, then we know that $s(n)=1$. As we know that for every $k\in [i:j-1]$, $1\leq i\leq j\leq n$, we have that either $s(k)<d(i,k,j)$ or $s(j)<d(i,j,j)-1$, we need to only check the cases when $i=j=n+1$ and when $j=n+1$ and $i<n+1$. For when $i=j=n+1$, we have that $s(n)=1<d(n,n,n)=3$. For $i<n+1$, we have that $s(n)=1<2\leq d(i,n,n)$. Thus, our conditions are satisfied and that $\mathrm{supp}\mbox{ } p_{n,-}(t)t_n\subseteq S(BW_{n+1})$. We know have that $\mathrm{supp}\mbox{ } p_{n+1,-}(t)\subseteq S(BW_{n+1})$. To prove that this inclusion is actually an equality, let us assume that the inclusion is strict and find a contradiction. If our inclusion is strict, then there exists an $s\in S(BW_{n+1})$ which is not in $\mathrm{supp}\mbox{ } p_{n+1,-}(t)$. We then must have that $$m_s(D) p_{n+1,-}(t)|_{t=0}=m_s(D)[p_{n,-}(t)(1+t_n)]|_{t=0}=0.$$ This would then mean that $m_s(D)p_{n,-}(t)|_{t=0}=0$, a contradiction to our induction hypothesis. Thus we must have equality. And in particular, we have $|S_-(BW_n)|=2^{n-1}$.
\end{proof}

Note that, while $q_n(t)$ and $p_{n,0}(t)$ are both homogeneous polynomials of degree $n$ in $n$ variables, their support is almost disjoint: $$\mathrm{supp}\mbox{ } p_{n,0}(t) \cap \mathrm{supp}\mbox{ } q_n(t) = \{(1,...,1)\}.$$ This observation is key to the proof of theorem \ref{t2}. As for the internal parking functions of $BW_n$, we have $$\left\vert{\{s\in S_-(BW_{n+1}):|s|=j\}}\right\vert={n-1 \choose j}, \mbox{ } 0\leq j\leq n-1.$$ This observation is key to the proof of theorem \ref{t1}.\

We will see that the zonotopal algebra of $BW_n$ hinges on the parking functions of $BW_n$. The Hilbert series presented in the next section, the monomial bases for the $\mathcal{P}$-central and $\mathcal{P}$-internal spaces, and the results connecting to the Stanley-Pitman polytope are all framed in terms of the parking functions of $BW_n$.

\section{The Zonotopal Algebra of the Broken Wheel Graph}\label{zonotopeBWG}

The zonotopal algebra of a graph consists of three pairs of polynomial spaces: a central pair, an internal pair, and an external pair. We will discuss the central and internal pairs of spaces for the broken wheel graph, and not the external pair as it does not play a role in our study. We will discuss the central Dahmen-Micchelli space $\mathcal{D}(X_n)$ of $BW_n$ and its dual, the $\mathcal{P}$-central space $\mathcal{P}(X_n)$. We will observe that $\mathcal{P}(X_n)$ is \textit{monomial}; i.e. has a monomial basis. Postnikov and Shapiro \cite{2003math......1110P} show that the monomial basis for $\mathcal{P}(X_n)$ must be given by the parking functions: $$\{m_s : s\in S(X_n)\}.$$ We will show that the volume polynomial $q_n(t)$ of the the Stanley-Pitman polytope lies in $\mathcal{D}(X_n)$, and that $q_n(t)$ is precisely the polynomial in a particular basis of $\mathcal{D}(X_n)$ which corresponds to the monomial $t_1t_2\cdots t_n$ in the monomial basis of $\mathcal{P}(X_n)$. Theorem \ref{t2} follows from this observation. We also show that once we reverse the order of the variables in $q_n(t)$, $\bar{q}_n(t_1,...,t_n):=q_n(t_n,...,t_1),$ the polynomial $\bar{q}_n(t)$ lies in the internal zonotopal space $\mathcal{D}_-(X_{n+1})$. At the same time, the internal zonotopal space $\mathcal{P}_-(X_{n+1})$ is monomial, with its monomial basis necessarily determined by the internal parking functions $$\{m_s : s\in S_-(X_{n+1})\}.$$ Theorem \ref{t1} follows from this observation. But in order to define and discuss these spaces in detail, we must first discuss the Tutte polynomial and Hilbert series of the broken wheel graph.

\subsection{The Tutte Polynomial and Hilbert Series of the Broken Wheel Graph}

Let $X$ be the corresponding matrix of a graph. Recall that the collection of its spanning trees $\mathbb{B}(X)$ correspond to the $n\times n$ invertible submatrices of $X$. We now define two valuations on the set $\mathbb{B}(X)$ that are the reversal of the external activity and internal activity as defined by Tutte.\

Both valuations require an ordering on $X$, we use the above-defined order $\prec: x_i \prec x_j$ if and only if $i<j$. Given $B\in \mathbb{B}(X)$, its \Dfn{valuation} is defined by $$val(B):=\left\vert{\{x\in (X\backslash B) : \{x\}\cup \{b\in B : b\prec x\} \mbox{ is independent } (\mbox{in } \mathbb{R}^n)\}}\right\vert.$$ Its \Dfn{dual valuation} is then defined as $$val^*(B):=\left\vert{\{b\in B : \{B\backslash b\}\cup\{x\in X\backslash B : b\prec x\} \mbox{ spans } \mathbb{R}^n\}}\right\vert.$$ The \Dfn{Tutte polynomial} is defined as the following bivariate polynomial, in the variables $s$ and $t$: $$T_{X}(s,t):= \sum_{B\in \mathbb{B}(X)} s^{n-val(B)}t^{n-val^*(B)}.$$

\begin{proposition}\label{p3}
The Tutte polynomial $T_{X_n}(s,t)$ of the broken wheel graph $BW_n$ is symmetric: $$T_{X_n}(s,t)=T_{X_n}(t,s).$$
\end{proposition}

\begin{proof}
Let $A$ be the $2n\times \mathbb{Z}$ matrix whose first row has entries $$a(1,u):=\left\{\begin{array}{cc}1 & j=1,2 \\-1 & j=3 \\0 & \mbox{Otherwise}\end{array}\right.$$ and whose entries are $a(i,j):=a(1,i+j-1)$ everywhere else. Note that each even row of this matrix is orthogonal to all the odd rows. We can see that $X_n$ is the submatrix of $A$ that corresponds to the rows indexed by $1,3,...,2n-1$ and columns $1,...,2n$. Let $Y_n$ be the matrix which has the same columns as $X_n$, but the complementary set of rows. The rows of $Y_n$ are still orthogonal to entries of $X_n$. Moreover, the matrix $Y_n$ is obtained from $X_n$ by performing the following operations: 
\begin{description}
\item[(i)] Multiply by -1 each odd column $x_{2i-1}$,
\item[(ii)] Reverse the order of the columns.
\end{description}
Thus $Y_n$ represents the same graph as $X_n$, with respect to a reverse order of the edges. Since the Tutte polynomial is invariant to the ordering of the edges, $T_{X_n}(s,t)=T_{Y_n}(s,t)$. On the other hand, since the row span of $Y_n$ is orthogonal to the row span of $X_n$, $Y_n$ is isomorphic to the dual matroid of $X_n$. It is further known that for every matroid $X$ with dual $\hat{X}$, $T_{X}(s,t)=T_{\hat{X}}(t,s)$. And so we have $$Y_{X_n}(s,t)=T_{Y_n}(s,t)=T_{\hat{X_n}}(s,t)=T_{X_n}(t,s).$$
\end{proof}

A spanning tree $B\in \mathbb{B}(X)$ is called \Dfn{internal} if $val^*(B)=n$ and \Dfn{maximal} if $val(B)=n$. Note that the number of internal trees of a graph $X$ equals $T_X(1,0)$, while the number of maximal trees equals $T_X(0,1)$.

\begin{proposition}\label{p4}
The Tutte polynomial of the broken wheel $BW_n$ satisfies $$T_{X_n}(1,0)=T_{X_n}(0,1)=2^{n-1}.$$
\end{proposition}

\begin{proof}
Let's consider first the set $\mathbb{B}_{max}(X_n)$ of maximal trees. For every $B\in \mathbb{B}_{max}(X_n)$, it follows directly from the definition that $x_1\notin B$, while $x_{2n}\in B$. After removing $x_{2n}$ from each maximal basis, we obtain a modified set $\mathbb{B}'_{max}(X_n)$. It is impossible that $\{x_{2i},x_{2i+1}\}\subset B$ for some $1\leq i< n$, since $x_{2i}+x_{2i+1}=x_{2i+2}$. Thus, we have $$\mathbb{B}'_{max}(X_n)\subset \times_{i=1}^{n-1}\{x_{2i},x_{2i+1}\},$$ and in particular $$\left\vert{\mathbb{B}_{max}(X_n)}\right\vert\leq 2^{n-1}.$$ Consider now the set $\mathbb{B}_-(X_n)$ of internal trees. The definition of an internal tree implies directly that $x_{2n}\notin B$, hence that $x_{2n-1}\in B$, for every internal basis. Removing $x_{2n-1}$ from each internal basis, we obtain the set $\mathbb{B}'_-(X_n)$. Now, consider the cross product $$A:=\times_{i=1}^{n-1}\{x_{2i-1},x_{2i}\}.$$ If we append $x_{2n-1}$ to any set in $A$, we obtain a basis $B\in \mathbb{B}(X_n)$; by induction on $j$, this follows from the assertion that every forest in $\times_{i=1}^j \{x_{2i-1},x_{2i}\}$, $1\leq j\leq n-1$ is a spanning tree on the subgraph that corresponds to the vertices $0,...,j$. 

Now, let $B=(b_1\prec b_2\prec\cdots \prec b_{n-1})$ be a tree in $A$. If $b_i=x_{2i-1}$, then $B\backslash b_i$ is completed to a spanning tree by $x_{2i}$, since $(b_1,...,b_{i-1})$ connects the vertices $0,...,i-1$ and each of $x_{2i-1}$ and $x_{2i}$ connects this vertex set to the vertex $i$. 

If $b_i=x_{2i}$, since $J:=(b_1,...,b_{i-1})$ is a spanning tree of $0,...,i-1$, the union $I\cup J\cup \{x_{2n-1}\}$ is full-rank, and hence $b_i$ is not internally active in $B$. Thus, $A\subset \mathbb{B}'_-(X),$ and $\left\vert{\mathbb{B}_-(X_n)}\right\vert\geq 2^{n-1}$. This completes the proof, since the symmetry of the Tutte polynomial implies that $\left\vert{\mathbb{B}_{max}(X_n)}\right\vert=\left\vert{\mathbb{B}_-(X_n)}\right\vert$.
\end{proof}

\begin{corollary}\label{c1}
We have that $$\mathbb{B}_{max}(X_n)=\times_{i=1}^{n-1}\{x_{2i},x_{2i+1}\}\times\{x_{2n}\},$$ and $$\mathbb{B}_-(X_n)=\times_{i=1}^{n-1}\{x_{2i-1},x_{2i}\}\times\{x_{2n-1}\}.$$
\end{corollary}

It is known \cite{2003math......1110P} that the number of parking functions of any graph $G$ equals the number of spanning threes of that graph: $$\left\vert{S(G)}\right\vert = \left\vert{\mathbb{B}(X_n)}\right\vert.$$ The \Dfn{central Hilbert series} $h_n := h_{X_n}$ is defined as $$h_n(j):=\left\vert{\{B\in \mathbb{X} : \mbox{val}(B)=j\}}\right\vert.$$ The Tutte polynomial determines $h_n$; i.e $h_n$ records the coefficients of $T_n(t,1)$ (in reverse enumerations). Note that proposition \ref{p4} asserts thus that $h_n(n)=2^{n-1}.$ Parking functions could also be used to determine $h_n$:

\begin{proposition}[Holtz and Ron, \cite{parkingfunctions}]\label{r2}
For each $0\leq j\leq n$, $$h_n(j)=\left\vert{\{s\in S(BW_n) : |s|:=\sum_{i=1}^n s(i)=j\}}\right\vert.$$ Thus there must be exactly $2^{n-1}$ parking functions with $|s|=n$.
\end{proposition}

\begin{corollary}\label{c2}
The maximal parking functions $S_{max}(BW_n)$ are exactly the subset $N$ of $\mathbb{N}^n$ of all sequences $s$ that can be written as a sum $$s=e_1+\sum_{j=1}^{n-1} a_j,$$ with $(e_i)_{i=1}^n$ the standard basis for $\mathbb{N}^n$, and $a_j\in \{e_j,e_{j+1}\}$ for every $j$.
\end{corollary}

\begin{proof}
From proposition \ref{p4}, we know that the number of parking functions $s$ with $|s|=n$ is $h_n(n)=2^{n-1}$. Since $\left\vert{N}\right\vert=2^{n-1}$, we merely need to verify that $N\subset S_{max}(BW_n)$. The fact that $N\subset \{0,1,2\}^n$ is clear, and so is the fact that $s(n)\leq 1$ for $s\in N$. Now, suppose that $s\in N$ and $s(j)=2$. Then $a_j=e_j$, and hence $\sum_{i=j+1}^{n-1} s(i)=\left\vert{[j+1 : n-1]}\right\vert<n-j$, which means that $s(k)=0$ for some $k>i$. Finally, if $s(j)=s(i)=2$ for some $j<i<n$, then $a_j=e_j$, while $a_{i-1}=e_i$. Hence $\sum_{k=j+1}^{i-1} s(k) =\left\vert{[j+1 : i-2]}\right\vert<j-i-1$, meaning that $s$ must vanish in between $j$ and $i$. Thus, $s$ is a parking function, and our claim follows.
\end{proof}

\begin{example}\label{e2}
The maximal parking functions of $BW_3$ are $$e_1+e_1+e_2=(2,1,0), e_1+e_1+e_3=(2,0,1), e_1+e_2+e_2=(1,2,0), \mbox{ and } e_1+e_2+e_3=(1,1,1).$$ 
\end{example}

Recall the set $\mathbb{B}_-(X_n)$ of internal bases. When restricting the valuation function to the internal bases, we obtain the \Dfn{internal Hilbert series} $$h_{n,-}(j):=\left\vert{\{B\in \mathbb{B}_-(X_n) : \mbox{val}(B)=j\}}\right\vert.$$ This function is also recorded by the Tutte polynomial and it is completely computable via the \textit{internal parking functions} of $BW_n$. It is known that the cardinality of the set of internal parking functions agrees with the number of internal bases, hence $$\left\vert{S_{-}(BW_n)}\right\vert=2^{n-1}.$$ More concretely,

\begin{corollary}\label{c3}
We have that $S_{-}(BW_n)=\{s\in \{0,1\}^n : s(n)=0\}.$
\end{corollary}

\begin{proof}
Since both sets above have the same cardinality $2^{n-1}$, we only need to check that every internal parking function must lie in $S_{-}(BW_n)$. Let $s$ be internal. Since $d(n,n,n)=2$, we conclude that $s(n)=0$. Since $d(i,i,i)=3$, for $i<n$, we conclude that $d(i)\leq 1$.
\end{proof}

It is known, \cite{HR} that the internal Hilbert series is graded by the internal parking functions, $$h_{n,-}(j)=\{s\in S_{-}(BW_n) : |s|=j\}.$$ We therefore conclude:

\begin{theorem}\label{t3}
The internal Hilbert series of $X_n$ is binomial: $$h_{n,-}(j)=\left\{\begin{array}{cc}{n-1 \choose j}, & 0\leq j<n \\0, & \mbox{otherwise}\end{array}\right.$$
\end{theorem}

We will now discuss zonotopal algebra in general, so that we have the framework for understanding the zonotopal algebra of the broken wheel graph, as well as the zonotopal algebra of the generalized broken wheel graph section \ref{GBWG}.

\subsection{Zonotopal Spaces}\label{introsonotopal}

Let $X$ be a matrix whose columns lie in $\mathbb{R}^n\setminus 0$ and span $\mathbb{R}^n$. We can consider two families of variable convex (bounded) polytopes: $$\Pi_{r}(t):=\{r: Xr=t, r\in \mathbb{R}_{\geq0}\} \text{ and } \Pi_{r}^1(t):=\{r: Xr=t, r\in [0,1]^n\}.$$ The \Dfn{box spline} $B_r(t)$ is the volume of $\Pi_r^1(t)$. As discussed in \cite{de1993box}, $B_r(t)$ is a piecewise polynomial. With $\mathbb{K}$ a field of characteristic zero, the \Dfn{central Dahmen-Micchelli space}, or central $\mathcal{D}$-space, $\mathcal{D}(X)$ of $B_r(t)$ is the vector space in $\mathbb{K}[t_1,...,t_n]$ generated by all polynomials in $B_r(t)$ and their partial derivatives. 

Viewing $X$ as a matroid whose ground set is the columns of $X$, $\mathcal{D}(X)$ can also be defined as the \textit{Macaulay inverse system} \cite{opac-b1117856} of a certain ideal $\mathcal{J}(X)$. To define this ideal, first note that a vector $r\in \mathbb{R}^n$ written in the basis $(t_1,...,t_n)$ naturally defines the polynomial $p_r=\sum_{i=1}^n \lambda_it_i$ in $\mathbb{K}[t_1,...,t_n]$; if $R$ is a set of vectors, then let $p_R:=\prod_{r\in R} p_r\in \mathbb{K}[t_1,...,t_n]$.  The ideal $\mathcal{J}(X)$ is generated by the polynomials in $\mathbb{K}[t_1,...,t_n]$ defined by the cocircuits of $X$: $$\mathcal{J}(X):=\ideal \{p_C : C\subseteq X \text{ cocircuit}\}\subseteq \mathbb{K}[t_1,...,t_n].$$ We then have that $$\mathcal{D}(X)=\ker \mathcal{J}(X) := \{f\in \mathbb{K}[t_1,...,t_n] : p(\frac{\delta}{\delta t_1},...,\frac{\delta}{\delta t_n})f=0\},$$ where $p$ runs over a set of generators of $\mathcal{J}(X)$. It was shown in \cite{Jia1985} that the dimension of $\mathcal{D}(X)$ is $\left\vert\mathbb{B}(X)\right\vert$, where $\mathbb{B}(X)$ is the set of bases of $\mathbb{R}^n$ which can be selected from $X$. Note that we use the same notation here as we did in the section above for the spanning trees of a matrix define by a graph, as when dealing with a graphical matroid (as we are), these sets are the same.

The \Dfn{central $\mathcal{P}$-space} of $X$ is defined as $$\mathcal{P}(X):=\Span \{p_R : R\subseteq X, X\setminus R \text{ has full rank}\}\subseteq \mathbb{K}[t_1,...,t_n].$$ $\mathcal{P}(X)$ can also be expressed as a Macaulay inverse system of a \textit{power ideal} generated by products of linear forms defining particular hyperplanes defined by $X$; see \cite{deboor1991} for more details. As proven in \cite{DR1990}, the central $\mathcal{D}(X)$ and $\mathcal{P}(X)$-spaces are dual under the pairing $\langle \cdot, \cdot \rangle: \mathcal{D}(X)\rightarrow \mathcal{P}(X)^*, f\mapsto \langle \cdot, f \rangle$, giving us that their Hilbert series are equal.

There are two more dual pairs which make up the zonotopal algebra of $X$. In order to define these pairs, we must define the set of \textit{internal bases} $\mathbb{B}_-(X)$ and the set of \textit{external bases} $\mathbb{B}_+(X)$ of $X$. Let $B_0=(b_1,...,b_n)$ be an arbitrary basis for $\mathbb{R}^n$ which is not necessarily contained in $\mathbb{B}(X)$. Let $X'=(X,B_0)$ and let $$ex: \{I\subseteq X : I \text{ linearly independent}\}\rightarrow \mathbb{B}(X')$$ be the function mapping an independent set in $X$ to its greedy extension in $X'$; i.e. for such an $I$, the vectors $b_1,...,b_n$ are added successively to $I$ unless the resulting set would be linearly dependent to get its image under $ex$. The set of \Dfn{external bases} $\mathbb{B}_+(X)$ is then defined as $$\mathbb{B}_+(X):=\{B\in \mathbb{B}(X') : B=ex(I) \text{ for some } I\subseteq X \text{ independent}\},$$ and the set of \Dfn{internal bases} $\mathbb{B}_-(X)$  is defined as $$\mathbb{B}_-(X):=\{B\in \mathbb{B}(X) : B \text{ contains no internally active elements}\}.$$ Note that the sets $\mathbb{B}_-(X)$ and $\mathbb{B}_+(X)$ as defined in the section above are equal to these sets for graphical matroids. We then have the following objects which define the \Dfn{internal $\mathcal{D}_-$-space} and \Dfn{external $\mathcal{D}_+$-space} of $X$: $$\mathcal{J}_-(X):=\ideal \{p_C : C\subseteq X \mathbb{B}_-(X)\text{-cocircuit}\}\subseteq \mathbb{K}[t_1,...,t_n],$$ $$\mathcal{D}_-(X):=\ker \mathcal{J}_-(X)\subseteq \mathbb{K}[t_1,...,t_n],$$ $$\mathcal{J}_+(X):=\ideal \{p_C : C\subseteq X \mathbb{B}_+(X)\text{-cocircuit}\}\subseteq \mathbb{K}[t_1,...,t_n],$$ $$\mathcal{D}_+(X):=\ker \mathcal{J}_+(X)\subseteq \mathbb{K}[t_1,...,t_n],$$ where a $\mathbb{B}_-(X)$-cocircuit (or $\mathbb{B}_+(X)$-cocircuit) is a subset of $X$ that intersects all bases in $\mathbb{B}_-(X)$ (or $\mathbb{B}_+(X)$), which is inclusion-minimal with this property. We then have that the \Dfn{internal $\mathcal{P}_{-}$-space} and \Dfn{external $\mathcal{P}_+$-space} of $X$ are defined as: $$\mathcal{P}_-(X):=\Span \{p_Y : Y\subseteq X\} \text{ and } \mathcal{P}_+(X):=\bigcap_{x\in X} \mathcal{P}(X\setminus x).$$ These three pairs of spaces make up the study of zonotopal algebras, and are discussed in great detail by Holtz and Ron in \cite{HR}. Now that we are familiar with their general definitions, we are ready to specialize our discussion to the case of the broken wheel graph.

\subsection{The Zonotopal Spaces of the Broken Wheel Graph}

We will now construct the zonotopal spaces associated to $X_n$. With $\mathbb{K}$ a field of characteristic zero, let $\mathbb{K}[t_1,...,t_n]_j$ be the subspace of $\mathbb{K}[t_1,...,t_n]$ consisting of homogeneous polynomials of degree $j$. Per \cite{HR}, each graph is associated with three pairs of subspaces of $\mathbb{K}[t_1,...,t_n]$: a central pair, an internal pair, and an external pair. As mentioned before, we will not need and hence will not introduce, the external pair. We will first introduce the central and internal Dahmen-Micchelli zonotopal spaces $\mathcal{D}(X_n)$ and $\mathcal{D}_-(X_n)$, respectively. We would like to stress that the latter space depends on the ordering we impose on the edges of the graph. The definition we give below corresponds to ordering the edges of $X_n$ in a reverse ordering. In fact, some of the proofs in this paper may be simplified once we use the reverse ordering. However, this reverse ordering is not inductive; the index of a given edge in the graph depends not only on the vertices that are connected, but also on the rank of the graph. To this end, we single out, for $1\leq i\leq j<n$, the following subset of $X_n$: $$X_{i,j,n}:=\{x_{2i},...,x_{2j}\}\cup \{x_{2i-1},x_{2j+1}\}.$$ For $j=n$, the definition is as follows: $$X_{i,n,n}:=\{x_{2i},...,x_{2j}\}\cup \{x_{2i-1}\}.$$

The \Dfn{central Dahmen-Micchelli space} $\mathcal{D}(X_n)$ is defined as the space of all polynomials in $\mathbb{K}[t_1,...,t_n]$ that are annihilated by each of the following differential operators: $$p_{X_{i,j,n}}(D), 1\leq i\leq j\leq n.$$ The \Dfn{internal Dahmen-Micchelli space} $\mathcal{D}_-(X_n)$ is defined as the space of all polynomials in $\mathbb{K}[t_1,...,t_n]$ that are annihilated by each of the following differential operators: $$p_{x_{2i}}(D)p_{x_{2i+1}}(D), 1\leq i<n, \mbox{ and } p_{x_{2n}}(D).$$

Note that these definitions are derived by considering all polynomials $p_C$, where $C$ is a cocircuit of $X_n$, and considering all differential operators which annihilate theses polynomials. This is the very construction of the central Dahmen-Micchelli space of $X_n$.

\begin{example}\label{e3}
The differential operators which define $\mathcal{D}(X_2)$ correspond to the subsets $$\{x_3,x_4\}, \{x_1,x_2,x_3\},\{x_1,x_2,x_4\}.$$ Those which correspond to $\mathcal{D}_-(X_3)$ are $$\{x_6\},\{x_4,x_5\},\{x_2,x_3\}.$$ Thus, while both spaces consist of polynomials in the variables $t_1$ and $t_2$ of degree not exceeding 2 the spaces themselves are different. Incidentally, the polynomial $t_2^2/2+t_1t_2$ lies in the first, while $t_1^2/2+t_1t_2$ lies in the second.
\end{example}

Next we will introduce the space dual to the central Dahmen-Micchelli space, called the $\mathcal{P}$-central space, and the space dual to the internal Dahmen-Micchelli space, called the $\mathcal{P}$-internal space. Here, and elsewhere, we denote by $$p_x(D), x\in \mathbb{R}^n,$$ the directional derivative in the $x$ direction. Also, for $Y\subset X$, $$p_Y:=\prod_{x\in Y} p_x.$$ The \Dfn{$\mathcal{P}$-central space} $\mathcal{P}(X_n)$ is the space of all polynomials in $\mathbb{K}[t_1,...,t_n]$ that are annihilated by each of the following differential operators: $$p_{1_{i,j}}(D)^{j-i+3}, 1\leq i\leq j<n,$$ and $$p_{1_{i,n}}(D)^{n-i+2}, 1\leq i<n.$$ Where $1_{i,j}:=e_i+e_{i+1}+\cdots + e_j,$ and $p_{1_{i,j}}(D)^k$ is $k$-fold differentiation in the $i_{i,j}$ direction. The \Dfn{$\mathcal{P}$-internal space} $\mathcal{P}_-(X_n)$ is the space of all polynomials in $\mathbb{K}[t_1,...,t_n]$ that are annihilated by each of the following differential operators: $$p_{1_{i,j}}(D)^{j-i+2}, 1\leq i\leq j<n,$$ and $$p_{1_{i,n}}(D)^{n-i+1}, i\leq i\leq n.$$

Note that the set of differential operators given in the definition of the $\mathcal{P}$-internal space is redundant. However, we defined it in this way to demonstrate the parallels to the central case definition. This also makes it easier to check that the definition is consistent with the general definition of the internal space, as given in \cite{HR}.\\

The Hilbert series of the broken wheel graph is captured by the homogeneous dimensions of the zonotopal spaces:

\begin{proposition}[Holtz and Ron, \cite{HR}]\label{r3}
For each $j\geq 0$, we have $$h_n(j)=\mbox{dim} (\mathcal{P}(X_n)\cap \mathbb{K}[t_1,...,t_n]_j),$$ and $$h_{n,-}(j)=\mbox{dim}(\mathcal{P}_-(X_n)\cap \mathbb{K}[t_1,...,t_n]_j).$$
\end{proposition}

Recall that a polynomial space is \Dfn{monomial} if it is spanned by monomials. The general theory of zonotopal algebra implies that once a $\mathcal{P}$-space of a graph is monomial, the corresponding parking functions yield a monomial basis for the space. This is exactly the case here.

\begin{theorem}\label{t4}
The zonotopal spaces $\mathcal{P}(X_n)$ and $\mathcal{P}_-(X_n)$ are monomial. Consequently, a basis for $\mathcal{P}(X_n)$ is given by the monomials $$m_s: t\mapsto t^s, s\in S_n,$$ while a basis for $\mathcal{P}_-(X_n)$ is given by the square-free monomials in the first $n-1$ variables.
\end{theorem}

\begin{proof}
We simply verify that each of the aforementioned monomials is annihilated by each of the requisite differential operators. The rest follows from proposition \ref{r3}.\

Let $s\in S_{-}(BW_n)$, and choose $1\leq i\leq n$. Since $m_s$ does not involve the variable $t_n$, $m_s$ is a polynomial of degree $\leq n-i$ in variables $t_i,...,t_n$; hence, it is annihilated by $p_{1_{i,n}}(D)^{n-i+1}$. Now choose $1\leq i\leq j<n$. Then $m_s$ is a polynomial of degree $\leq n-i+1$ in the variables $t_i,...,t_j$; hence, it is annihilated by $p_{1_{i,j}}(D)^{n-i+2}$. This completes the proof for the internal case.\

Assume now that $s\in S_n$. Note that the characterization of $s$ implies that $\sum_{j=i}^n s(j)\leq n-i+1$ (since the number of 2-entries on $[i:n]$ cannot exceed the number of 0 entries), while $\sum_{j=i}^k s(j)\leq n-i+2$. Thus, an analogous argument to the above yields the result.
\end{proof}

We note in passing that the $\mathcal{P}$-external space is \textit{not} monomial. In fact, external zonotopal spaces are never monomial unless the underlying linear matroid in the tensor of rank-1 matroids.\\

The general theory of zonotopal algebra tells us that the central spaces form a dual pair, and that the same is true for the internal pair. To this end, we make the following definition: Let $X$ be a graph, and $s$ a parking function of $X$. A polynomial $p\in \mathbb{K}[t_1,...,t_n]$ is called \Dfn{$\mathbf{s}$-monic} in $X$ if $p\in \mathcal{D}(X)$, the monomial $m_s$ appears in the monomial expansion of $p$ with coefficient 1, and all other monomials $m_{s'}$ that correspond to the other parking functions of $X$ appear with coefficient 0 in this expansion.\

Similarly, for an internal parking function $s$ of $X$, $p\in \mathbb{K}[t_1,...,t_n]$ is \Dfn{internally $\mathbf{s}$-monic} in $X$ if $p\in \mathcal{D}_-(X)$ (for the fixed ordering of $X$ that is considered), $m_s$ appears in the monomial expansion of $p$ with coefficient 1, and all other monomials $m_{s'}$ that correspond to the other internal parking functions appear with coefficient 0.

\begin{proposition}[Holtz and Ron, \cite{HR}]\label{r4}
Let $X$ be a graphic matroid, and assume that $\mathcal{P}(X)$ is monomial. Then, for each parking function $s$ of $X$ there exists a unique $s$-monic polynomial in $X$. Similarly, if $\mathcal{P}_-(X)$ is monomial, and $s$ is an internal parking function, there exist a unique internal $s$-monic polynomial in $X$. The collection of all $s$-monic polynomials in $\mathcal{D}(X)$ form a basis for $\mathcal{D}(X)$ (which is dual to the monomial basis of $\mathcal{P}(X)$); similarly for $\mathcal{D}_-(X)$.
\end{proposition}

\begin{corollary}\label{c4}
For each broken wheel graph $BW_n$, there is a unique basis for $\mathcal{D}(X_n)$ which is monic in $X_n$. Similarly, there is a unique basis for $\mathcal{D}_-(X_n)$ which is internally monic in $X_n$.
\end{corollary}

\begin{example}\label{e4}
In example \ref{e3}, the polynomial $t_2^2/2+t_1t_2$ is (1,1)-monic in $X_2$ and $t_1^2+t_1t_2$ is internally (1,1,0)-monic in $X_3$.
\end{example}

\section{The Stanley-Pitman Polytope}\label{SPpolytope}

Pitman and Stanley \cite{PS} studied the $n$-dimensional polytope $$Q_n(t):=\{r\in \mathbb{R}_+^n : \sum_{i=j}^n r_i\leq \sum_{i=j}^n t_i, 1\leq j\leq n\},$$ and outlined several of its properties as well as found an explicit expression for it's volume $$q_n(t):=\operatorname{vol}(Q_n(t)).$$ In this section, we will draw a connection between the Stanley-Pitman polytope $Q_n(t)$ and zonotopal algebra of the broken wheel graph as well as prove theorems  \ref{t1} and \ref{t2} from the introduction.

\subsection{Connecting to the Zonotopal Algbra of the Broken Wheel Graph}

We first need to introduce the additional variables $(u_1,...,u_n)$ such that, for each $j$, we have $$u_j+\sum_{i=1}^j r_i=\sum_{i=1}^j t_i.$$ Equivalently, $$u_j+r_j-u_{j-1}=t_j, \mbox{ } j=2,...,n,$$ and $$u_1+r_1=t_1.$$ We then observe that these equations are equivalent to $$X_na=t,$$ with the $2n$-vector $a$ obtained from the concatenated $u,r$ by a suitable permutation: $u_i$ corresponds to $a_{21-1}$, and $r_i$ corresponds to $a_{2i}$. We also have the ``side condition'' that $a\in \mathbb{R}_+^{2n}.$ With this, we have the conditions necessary to link the zonotopal algebra of the broken wheel graph with the Stanley-Pitman polytope. With this, we have that the volume polynomial $q_n(t)$ is a homogeneous polynomial of maximal degree $n$ in the zonotopal space $\mathcal{D}(X_n)$:

\begin{theorem}\label{t5}
The Stanley-Pitman volume polynomial $q_n(t)$ is the monic polynomial in $\mathcal{D}(X_n)$ that corresponds to the parking function $(1,...,1)\in \mathbb{R}^n$. In addition, it is also the unique internally monic polynomial of maximal degree in $\mathcal{D}_-(X_{n+1})$ which corresponds to the unique internal parking function in $X_{n+1}$ of maximal degree, viz $(1,1,...,1,0)\in \mathbb{R}^{n+1}$.
\end{theorem}

\begin{proof}
We have that $q_n(t)$ is the polynomial consisting of the sum of all the normalized monomials $\frac{m_s}{s!}$ of degree $n$ which satisfy \begin{equation}\label{qninequalities} \sum_{j=1}^i s(j)\leq i, 1\leq i\leq n.\end{equation} This can be seen by applying our notation to proposition \ref{r1}, the main theorem of Pitman and Stanley in \cite{PS}. Let $C_n$ be the set of vectors $s$ such that $\frac{m_s}{s!}$ is a term of $q_n(t)$. We notice that if $s\in C_n$ is a maximal parking function, then it must satisfy \begin{equation}\label{qninequalities2} i\leq \sum_{j=1}^i s(j)\leq i+1.\end{equation} We can see this via corollary \ref{c2}. Thus, we must have that $\sum_{j=1}^i s(j)=i$ for all $j$; in other words, $s=(1,...,1)$, making $q_n(t)$ the unique $(1,...,1)$-monic polynomial in $\mathcal{D}(X_n)$.\

In order to show that $q_n(t)$ is the unique internally monic polynomial of maximal degree in $\mathcal{D}_-(X_{n+1})$ which corresponds to the unique internal parking function in $X_{n+1}$ of maximal degree, we only need to show that $q_n(t)\in \mathcal{D}_-(X_{n+1})$, as our argument directly above gives us our correspondence between $q_n(t)$ and (1,...,1,0), up to normalization.\

We thus need to check that $q_n(t)$ is annihilated by the polynomials $$p_{x_{2i}}(D)p_{x_{2i+1}}(D), \mbox{ } 1\leq i< n+1, \mbox{ and } p_{x_{2(n+1)}}(D).$$ We can quickly see that $p_{x_{2(n+1)}}(D)$ annihilates $q_n(t)$ as $p_{x_{2(n+1)}}(D)$ is differentiation in the $t_{n+1}$ variable, of which there are none in $q_n(t)$. The other operators we need to consider are $$D_{i+1}D_i-D_i^2, i=1,...,n.$$ When $i=n$, we have that $(D_{n+1}D_n-D_n^2)q_n(t)=D_{n+1}D_nq_n(t)-D_n^2q_n(t)=0$, as the degree of $t_{n+1}$ is 0 and the degree of $t_n$ is either 0 or 1 for any term of $q_n(t)$.\

For $i<n$, let's consider a term $m_s/s!$ of $q_n(t)$. If $s(i)\leq 1$, we then have that $m_s/s!$ is annihilated by $D_{i+1}D_i-D_i^2$. If $s(i)\geq 2$, then let us prove that $s\in C_n$ if and only if $\hat{s}:=s-e_i+e_{i+1}\in C_n$, as we will then see that the annihilation of $q_n(t)$ by the differential operators in question will directly follow from this statement. First, let's assume that $s\in C_n$. We can then see that $$\sum_{j=1}^i s(j)=\sum_{j=1}^i \hat{s}(j)\leq i.$$ Thus, $\hat{s}$ satisfies the inequalities (\ref{qninequalities}), meaning that $\hat{s}\in C_n$. Now, let us assume that $\hat{s}\in C_n$, and let us further assume for contradiction that $s\notin C_n$. Then there exists some $i$ such that $\sum_{j=1}^i s(j)>i$. As $\hat{s}\in C_n$, we know that $$\sum_{j=1}^{i+1} s(j)=\sum_{j=1}^{i+1} \hat{s}(j)\leq i+1,$$ meaning that $i+1\leq s(1)+\cdots +s(i)\leq s(1)+\cdots + s(i)+s(i+1)\leq i+1$. This then means that $s(i+1)=0$ and that $\sum_{j=1}^i s(j)=i+1$. But this then means that $$\sum_{j=1}^i \hat{s}(j)=\sum_{j=1}^i s(j)=i+1\leq i,$$ a contradiction to our assumption that $\hat{s}\in C_n$. Thus, we must have that $s\in C_n$. Now, we can see, for $s\in C_n$ with $s(i)\geq 2$, that $$D_i^2(m_s/s!)=D_{i+1}D_i(m_{\hat{s}}/\hat{s}!)=m_{s-2e_i}/(s-2e_i)!$$ and thus that the $D_i^2(m_s/s!)$ term  and the $D_{i+1}D_i(m_{\hat{s}}/\hat{s}!)$  term in $(D_{i+1}D_i-D_i^2)q_n(t)$ cancel each other out. Furthermore, we know there exists a $t\in C_n$ (i.e. $t=s+e_i-e_{i+1}$) such that $s=\hat{t}$ and that $\hat{\hat{s}}\in C_n$. From this we know that the $D_i^2(m_t/t!)$ term and the $D_{i+1}D_i(m_{s}/s!)$  term cancel each other out, and that the $D_i^2(m_{\hat{s}}/\hat{s}!)$ term  and the $D_{i+1}D_i(m_{\hat{\hat{s}}}/\hat{\hat{s}}!)$ term cancel each other out. Carrying on in this fashion, we have that all of the terms of $(D_{i+1}D_i-D_i^2)q_n(t)$ are cancelled and we have that $(D_{i+1}D_i-D_i^2)q_n(t)=0$ as desired. We thus have that $q_n(t)\in \mathcal{D}_-(X_{n+1})$, giving us our result.
\end{proof}

\begin{corollary}\label{c4}
The polynomial space $\bar{q_n}(t)$ that is generated by the derivatives (of all orders) of the polynomial $q_n(t)$ is the zonotopal space $\mathcal{D}_-(X_{n+1})$. Thus, its homogeneous dimensions are binomial: $$\mbox{dim} (q_n(t)\cap \prod_j^0)={n \choose j}, \mbox{ } j=0,...,n.$$
\end{corollary}

\begin{proof}
We know from theorem \ref{t4} that $\mathcal{P}_-(X_{n+1})$ is generated by the square-free monomials in the first $n$ variables. Let's consider the generator of maximal degree, $t_1\cdots t_n$. As we take partial derivatives of all orders of this monomial, we can see that we will generate all square-free monomials of degree $\leq n$. Thus, we have that $\mathcal{P}_-(X_{n+1})$ is the space generated by the derivatives (of all orders) of the monomial $t_1\cdots t_n$.\

Via proposition \ref{r4}, we know that for every generator $m_s$ of $\mathcal{P}_-(X_{n+1})$, there is a corresponding generator of $\mathcal{D}_-(X_{n+1})$ which is the unique internal $s$-monic polynomial in $X_{n+1}$.

Now, let's consider the polynomial $q_n(t)$, and let $q'(t):=D_1^{k_1}\cdots D_n^{k_n}q_n(t)$. We then know that $q'(t) \in \mathcal{D}_-(X_{n+1})$ and that the square-free monomial $D_1^{k_1}\cdots D_n^{k_n} t_1\cdots t_n$ is a term of $q'(t)$. Let $s$ be the exponent vector of $D_1^{k_1}\cdots D_n^{k_n} t_1\cdots t_n$. Then we know that $s$ is an internal parking function, meaning that $q'(t)$ must be the unique internal $s$-monic polynomial and thus is also a generator of $\mathcal{D}_-(X_{n+1})$.\

Every generator of $\mathcal{D}_-(X_{n+1})$ is a derivative of $q_n(t)$, and every derivative of $q_n(t)$ is a generator of $\mathcal{D}_-(X_{n+1})$. Thus, we have that $\mathcal{D}_-(X_{n+1})$ is the polynomial space generated by the derivatives (of all orders) of the polynomial $q_n(t)$ as desired.
\end{proof}

\subsection{Proving Theorems \ref{t1} and \ref{t2} From the Introduction}

From theorem \ref{t5} and corollary \ref{c4}, the proofs of theorems \ref{t1} and \ref{t2} from the introduction become clear. Let us now prove these theorems. Recall that we denote partial differentiation with respect to $t_i$ by $D_i$; i.e. with $p_i : \mathbb{R}^n\rightarrow \mathbb{R}^n, t\mapsto t_i,$  we have $D_i:=p_i(D)$, and $D_0 := 0$.

\begin{reptheorem}{t1}
The polynomial $q_n(t)$ is the only polynomial (up to normalization) of degree $n$ that is annihilated by each of the following differential operators $$D_i(D_i-D_{i-1}), \mbox{ } i=1,...,n.$$ Moreover, let $\mathcal{P}_{n,j}$ be the subspace of homogeneous polynomials (in $n$-indeterminates) of degree $j$ that are annihilated by each of the above differential operators. Then:
\begin{enumerate}
\item $\mathcal{P}_{n,j}$ lies in the span of the translates of $q_n$.
\item dim $\mathcal{P}_{n,j}$ = ${n \choose j}$.
\end{enumerate}
\end{reptheorem}

\begin{proof}
We show in the proof of theorem \ref{t5} that $q_n(t)$ is the unique internally monic polynomial of maximal degree in $\mathcal{D}_-(X_{n+1})$. We also have that $q_n(t)$ lies in the dual central zonotopal space, $\mathcal{D}(X_n)$, meaning that $q_n(t)$ is annihilated by $D_i(D_i-D_{i-1}),$ for i=1,...,n, by definition. Corollary \ref{c4} of this note can be rephrased as saying that the space of translates of $q_n(t)$ is $\mathcal{D}_-(X_{n+1})$. We then can see that for a given degree $j$, we have that $\mathcal{P}_{n,j}\subset \mathbb{K}[t_1,...,t_n]_j$, giving \begin{equation}\label{incl}(\mathcal{D}_-(X_{n+1})\cap \mathcal{P}_{n,j})\subseteq (\mathcal{D}_-(X_{n+1})\cap \mathbb{K}[t_1,...,t_n]_j).\end{equation} In other words, we have that $\mathcal{P}_{n,j}\subset \mathcal{D}_-(X_{n+1})$; i.e. we have that $\mathcal{P}_{n,j}$ lies in the span of the translates of $q_n(t)$. Furthermore, we can actually see that our inclusion (\ref{incl}) is actually an equality, as the dimension of both $(\mathcal{D}_-(X_{n+1})\cap \mathcal{P}_{n,j})$ and $(\mathcal{D}_-(X_{n+1})\cap \mathbb{K}[t_1,...,t_n]_j)$ is ${n \choose j}$. 

The dimension of $(\mathcal{D}_-(X_{n+1})\cap \mathbb{K}[t_1,...,t_n]_j)$ is given to us by corollary \ref{c4}. The dimension of  $(\mathcal{D}_-(X_{n+1})\cap \mathcal{P}_{n,j})$ is gotten by counting the number of internal parking functions of the broken wheel graph.\

The uniqueness of $q_n(t)$ can then be quickly seen by the fact that $\mathcal{P}_{n,n}\subset \mathcal{D}_-(X_{n+1})$ and that $q_n(t)$ is the unique internally monic polynomial of maximal degree in $\mathcal{D}_-(X_{n+1})$.
\end{proof}

\begin{example}\label{e5}
Let's consider $n=2$. We then have that $q_2(t)=t_2^2/2+t_1t_2$. The theorem above then tells us that $q_2(t)$ is the only polynomial which is annihilated by $D_2(D_2-D_1)$ and $D_1^2$.
\end{example}

The following result gives another characterization of $q_n(t)$. It should be noted that while the following result resembles the theorem \ref{t1}, it is the result of a rather different observation.

\begin{reptheorem}{t2}
The polynomial $q_n(t)$ is the only polynomial $q(t)$ (in $n$ variables) that satisfies the following two properties:
\begin{enumerate}
\item With $m$ the square-free monomial $$m: t\mapsto \prod_{i=1}^{n} t_i,$$ the monomial support of $(q-m)(t)$ is disjoint of the monomial support of the polynomial $$t\mapsto \prod_{i=1}^n (t_i+t_{i-1}), \mbox{ } t_0 := 0.$$
\item $q(t)$ is annihilated by each of the following differential operators: $$(D_{j+1}-D_j)(\prod_{k=i}^j D_k)(D_i-D_{i-1}), \mbox{ } 1\leq i\leq < n$$ and $$ (\prod_{k=1}^n D_k)(D_i-D_{i-1}),\mbox{ } 1\leq i\leq n.$$
\end{enumerate}
\end{reptheorem}

\begin{proof}
For part $(1)$, we can see that the polynomial $t\mapsto \prod_{i=1}^n (t_{i-1}+t_i)$ is exactly the polynomial $p_{n,0}(t)$, introduced in proposition \ref{p1}. From proposition \ref{p1}, we know that $S_{max}(BW_n)=\mathrm{supp}\mbox{ } p_{n,0}(t)$. In other words, $$\mathrm{supp}\mbox{ } p_{n,0}(t)=\{s\in S(BW_n) : |s|=n\}.$$ We know from \cite{PS} that the only maximal parking function in $C_n$ which gives rise to a term $m_s/s!$ of $q_n(t)$ is $s=(1,...,1)$. Thus, we have that $\mathrm{supp}\mbox{ } p_{n,0}(t) \cap \mathrm{supp}\mbox{ } q_n(t) = \{(1,...,1)\}$. And so, by subtracting the monomial $m$ from $q_n(t)$, we can quickly see that $$\mathrm{supp}\mbox{ } p_{n,0}(t) \cap \mathrm{supp}\mbox{ } (q_n-m)(t) = \emptyset.$$ For part $(2)$, we again know that $q_n(t)$ must be eliminated by the differential operators in question as $q_n(t)$ lie in $\mathcal{D}(X_n)$. We can see that $q_n(t)$ is the only polynomial in $n$ variables which satisfies both $(1)$ and $(2)$ because, as $q_n(t)$ is uniquely the polynomial in $\mathcal{D}(X_n)$ which corresponds to the monomial $m$ in the basis of $\mathcal{D}(X_n)$, we have that $q_n(t)$ is the only polynomial satisfying (2) which also satisfies (1).
\end{proof}

\begin{example}\label{e6}
Considering $n=2$, we have that $(q_2-m)(t)=t_2^2/2$. We can then see that $\mathrm{supp}\mbox{ } (q_2-m)(t)=\{(0,2)\}$ and $\mathrm{supp}\mbox{ } t_1(t_1+t_2)=\{(2,0),(1,1)\}$ are disjoint, and thus that condition $(1)$ of the theorem above is satisfied. For condition $(2)$, we can see that $q_2(t)$ is annihilated by both $D_1^2(D_2-D_1)$ and $D_1^2$.
\end{example}

\begin{example}\label{e7}
For $n=3$, we have that $(q_3-m)(t)=t_3^2/6+t_3^2(t_1+t_2)/2+t_3t_2^2/2$. For condition $(1)$, we can see that $\mathrm{supp}\mbox{ } t_1(t_1+t_2)(t_2+t_3)=\{(2,1,0),(2,0,1),(1,2,0),(1,1,1)\}$ is disjoint from $\mathrm{supp}\mbox{ } (q_3-m)(t)=\{(0,0,3),(1,0,2),(0,1,2),(0,2,1)\}$. For condition $(2)$ we can see that $q_3(t)$ is annihilated by the operators $$D_3(D_3-D_2),\mbox{ } D_3D_2(D_2-D_1),\mbox{ } (D_3-D_2)D_2(D_2-D_1),\mbox{ } (D_2-D_1)D_1^2.$$
\end{example}

\subsection{A Polyhedral Subdivision Relating to the Associahedron}\label{assoc}

Pitman and Stanley \cite{PS} describe a polyhedral subdivision of $Q_n(t)$ closely related to the \textit{associahedron}. The \Dfn{associahedron} $\mathfrak{A}_n$ is a polytope whose vertices correspond to the triangulations of the $(n+3)$-gon and whose edges correspond to \Dfn{flips} of diagonal edges; i.e. removing one diagonal edge from a given triangulation and replacing it with another diagonal edge. This section is included as a review of this polyhedral subdivision of $Q_n(t)$ which Pitman and Stanley \cite{PS} present and how the volume of each polytope in their subdivision is captured by the zonotopal algebra of the broken wheel graph. This connection was the main inspiration for the generalized broken wheel graph appearing in the coming sections.

\begin{figure}[htbp]
\begin{center}
\includegraphics[scale=.4]{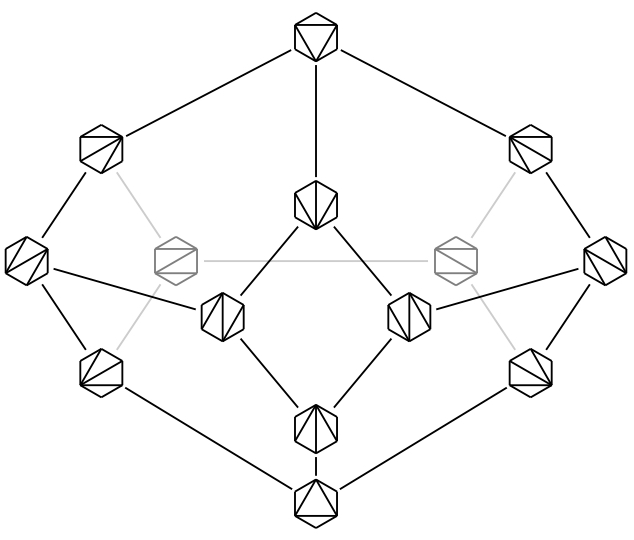}
\caption{The associahedron $\mathfrak{A}_3$.}
\label{default}
\end{center}
\end{figure}

Its dual is a simplicial complex whose vertices are diagonals of a convex $(n+3)$-gon, simplices are the partial triangulations of the $(n+3)$-gon, and whose maximal simplices are triangulations of the $(n+3)$-gon. Pitman and Stanley \cite{PS} construct a fan $F_n$ whose chambers are indexed by \textit{plane binary trees} with $n$ internal vertices and prove the following result:

\begin{proposition}[Pitman and Stanley \cite{PS}] The face poset of the fan $F_n$, with a top element adjoined, isomorphic to the dual $\mbox{dec}(E_{n+2})^*$ of the face lattice of the associahedron. 
\end{proposition} 

A \Dfn{plane binary tree} is a plane tree such that each vertex has zero or two substrees. If a vertex has zero subtrees, then we call it a \Dfn{leaf}, and if a vertex has two subtrees, then we call it an \Dfn{internal vertex}. The construction of the fan $F_n$ is as follows. First consider a binary tree $T$. Do a depth-first search of $T$, labelling its internal vertices 1 through $n$ in the order they are encountered from above. This labelling is referred to by Pitman and Stanley as the \Dfn{binary search labelling}. 

\begin{figure}[htbp]
\begin{center}
\begin{tikzpicture}
\tikzset{vertex/.style = {shape=circle,draw,minimum size=1.5em}}
\tikzset{edge/.style = {-,> = latex'}}
\node[vertex] (a) at  (-1,1) {1};
\node[vertex] (b) at  (-.5,2) {2};
\node[vertex] (c) at  (0,0) {3};
\node[vertex] (d) at  (1,1) {4};
\node[vertex] (e) at  (-1,3) {};
\node[vertex] (f) at  (0,3) {};
\node[vertex] (g) at  (.5,2) {};
\node[vertex] (h) at  (-1.5,2) {};
\node[vertex] (i) at  (1.5,2) {};
\draw[edge] (c) to (a);
\draw[edge] (c) to (d);
\draw[edge] (a) to (b);
\draw[edge] (a) to (h);
\draw[edge] (d) to (g);
\draw[edge] (d) to (i);
\draw[edge] (b) to (f);
\draw[edge] (b) to (e);
\end{tikzpicture}
\caption{A plane tree with the binary search labelling.}
\label{default}
\end{center}
\end{figure}
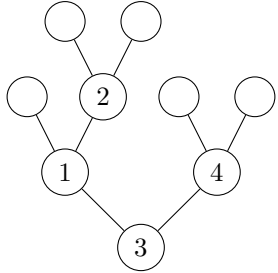

If an internal vertex of $T$ with label $i$ is covered by $j$, then associate to the pair $(i,j)$ the inequality $$x_{i+1}+x_{i+2}+\cdots + x_{j}\leq 0$$ if $i<j$ and the inequality $$x_{j+1}+x_{j+2}+\cdots + x_i\geq 0$$ if $i>j$. We then have a system of $n-1$ homogeneous linear equations which define a simplicial cone in $\mathbb{R}^{n-1}$. These cones, as they range over all plane binary trees with $n$ internal vertices, form the chambers of a complete fan, denoted $F_n$, in $\mathbb{R}^{n-1}$.\

Let $T\in \mathcal{T}_n$, where $\mathcal{T}_n$ is the set of binary trees with $n$ internal vertices. Pitman and Stanley \cite{PS} then construct sets $\triangle_T(x)$ which form the maximal faces of a polyhedral decomposition $\Gamma_n$ of $Q_n(x)$ whose set of interior faces, ordered by inclusion, is isomorphic to the face lattice of the dual associahedron. They also give the volume of these maximal faces.

\begin{proposition}[Pitman and Stanley \cite{PS}]\label{assocvolume} We have the following:
\begin{enumerate}
\item The sets $\triangle_T(x)$, $T\in \mathcal{T}_n$, form the maximal faces of a polyhedral decomposition $\Gamma_n$ of $Q_n(x)$.
\item Let $k(T)=(k_1,...,k_n)$, $T\in \mathcal{T}_n$. Then Vol($\triangle_T(x)$)$=\frac{x_1^{k_1}}{k_1!}\cdots \frac{x_n^{k_n}}{k_n!}$.
\item The set of interior faces of $\Gamma_n$, ordered by inclusion, is isomorphic to the face lattice of the dual associahedron.
\end{enumerate}
\end{proposition}

In order to understand this result, we must define the objects mentioned in it; let us do this. Given a plane tree $T$ and $E$ the set of edges of $T$, let's define a function $\ell : E\rightarrow \mathbb{R}_+$ sending every edge $e$ of $T$ to a positive real number $\ell(e)$. We will then call the pair $(T,\ell)$ a \Dfn{plane tree with edge lengths}. Now fix a real number $s>0$ which we would like to be the sum of the edge lengths of a plane tree. Let $x=(x_1,...,x_n)\in \mathbb{R}_+^n$ be such that $\sum x_i < s$ and $y=(y_1,...,y_n)\in \mathbb{R}_+^n$ with $$y_1+\cdots y_i\leq x_1+\cdots x_i, \mbox{ } 1\leq i\leq n.$$ For each pair $(x,y)$, we can assign a plane tree with edge lengths $\varphi(x,y)=(\overline{T}, \ell)$ as described in \cite[p. 32]{PS}. We start with a root and traverse the tree in depth-first order:
\begin{enumerate}
\item Go up distance $x_i$, then down distance $y_i$, for $1\leq i\leq n$.
\item Finish the tree by going up distance $x_{n+1}=s-x_1-\cdots - x_n$ and down distance $y_{n+1}=s-y_1-\cdots - y_n$.
\end{enumerate}
We then have a planted (i.e. the root as one child) plane binary tree with edge lengths. Let $T$ be the tree obtained by removing the roots and its incident edge from $\overline{T}$. Now let $x=(x_1,...,x_n)$ be a sequence with $\sum x_i <s$ and let $T\in \mathcal{T}_n$ be a plane binary tree without edge lengths. We define $$\triangle_T(x):=\{y\in \mathbb{R}_+^n : \varphi(x,y)=(\overline{T},\ell) \mbox{ for some } \ell\}.$$ For $T\in \mathcal{T}_n$, with the binary search labeling of its internal vertices, let $k(T)=(k_1,...,k_n)\in \mathbb{N}^n$ such that:
\begin{enumerate}
\item $k_i=0$ if the left child of vertex $i$ is an internal vertex.
\item If the left child of vertex $i$ is an endpoint, then let $k_i$ be the largest integer $r$ such that there exists a chain $i<j_1<\cdots j_r$ of internal vertices where $j_h$ is a left child of $j_{h+1}$ for $1\leq h\leq r-1$.
\end{enumerate}

Proposition \ref{assocvolume} tells us that the volume of every polytope in this particular subdivision of the Stanley-Pitman polytope $Q_n(t)$ is a term of $q_n(t)$, and that all terms of $q_n(t)$ appear as such volumes. So not only is the zonotopal algebra of the broken wheel graph capturing the volume of $Q_n(t)$, it is also capturing the volumes of the polytopes of a polyhedral subdivision of $Q_n(t)$ whose set of interior faces, ordered by inclusion, is isomorphic to the face lattice of the dual associahedron. This observation was our motivation for studying the volumes of polyhedral subdivisions in terms of zonotopal algebras and lead us to the generalized broken wheel graph.

\section{The Zonotopal Algebra of the Generalized Broken Wheel Graph}\label{GBWG}

While the zonotopal algebra of the broken wheel graph and its connection to the Stanley-Pitman polytope are rich in their own right, our study reaches even further. We will consider the zonotopal algebra of the \textit{generalized broken wheel graph} $GBW_n(T)$ over a tree $T$ with $n$ vertices and how it relates to the regular simplex $\mathfrak{Sim}_n(t_1,...,t_n)$ with positive parameters $(t_i)_{i\in [n]}$, defined by the inequalities $$\sum_{i=1}^n r_i\leq \sum_{i=1}^n t_i, \mbox{ } r_i\in \mathbb{R}_+^n,$$ where the $(r_i)_{i\in [n]}$ are variables. Since our set-up is homogeneous, we will assume without loss of generality that $$\sum_{i=1}^n t_i=1.$$ We will show how to  partition $\mathfrak{Sim}_n(t_1,...,t_n)$ into $2^{n-1}$ polytopes, where each polytope's volume is captured by the zonotopal algebra of $GBW_n(T)$. We begin by outlining the set-up necessary to define the generalized broken wheel graph.

\subsection{Constructing the Generalized Broken Wheel Graph} 

Our first step in this process is to enumerate all rooted trees with $n$ vertices. So, for example, there are two rooted trees with $3$ vertices, which we will respectively call the ``line tree'' and the ``fork tree", as illustrated in figure \ref{lineforktree}. For convenience, let's generally label the vertices of any rooted trees we consider $1$ through $n$ and always assume that the root of each tree is $1$.

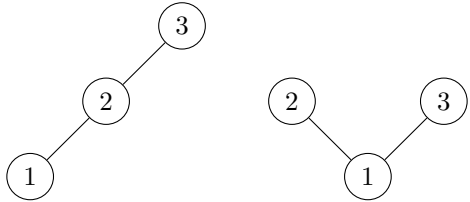
\begin{figure}[h!]
\begin{center}
\begin{tikzpicture}
\tikzset{vertex/.style = {shape=circle,draw,minimum size=1.5em}}
\tikzset{edge/.style = {-,> = latex'}}
\node[vertex] (b) at  (1,1) {1};
\node[vertex] (c) at  (2,2) {2};
\node[vertex] (d) at  (3,3) {3};
\draw[edge] (b) to (c);
\draw[edge] (c) to (d);
\end{tikzpicture}
\qquad
\begin{tikzpicture}
\tikzset{vertex/.style = {shape=circle,draw,minimum size=1.5em}}
\tikzset{edge/.style = {-,> = latex'}}
\node[vertex] (b) at  (0,1) {1};
\node[vertex] (c) at  (-1,2) {2};
\node[vertex] (d) at  (1,2) {3};
\draw[edge] (c) to (b);
\draw[edge] (d) to (b);
\end{tikzpicture}
\caption{The line tree (to the left) and the fork tree (to the right).}
\label{lineforktree}
\end{center}
\end{figure}

There are $2^{n-1}$ different ways to direct the edges of a rooted tree $T$. For $n=3$, we have four directed trees from the line tree and four from the fork tree, as illustrated in figure \ref{directedlineforktree}.

\begin{figure}[h!]
\begin{center}
\begin{tikzpicture}
\tikzset{vertex/.style = {shape=circle,draw,minimum size=1.5em}}
\tikzset{edge/.style = {->,> = latex'}}
\node[vertex] (b) at  (1,1) {1};
\node[vertex] (c) at  (2,2) {2};
\node[vertex] (d) at  (3,3) {3};
\draw[edge] (b) to (c);
\draw[edge] (c) to (d);
\end{tikzpicture}
\qquad
\begin{tikzpicture}
\tikzset{vertex/.style = {shape=circle,draw,minimum size=1.5em}}
\tikzset{edge/.style = {->,> = latex'}}
\node[vertex] (b) at  (1,1) {1};
\node[vertex] (c) at  (2,2) {2};
\node[vertex] (d) at  (3,3) {3};
\draw[edge] (b) to (c);
\draw[edge] (d) to (c);
\end{tikzpicture}
\qquad
\begin{tikzpicture}
\tikzset{vertex/.style = {shape=circle,draw,minimum size=1.5em}}
\tikzset{edge/.style = {->,> = latex'}}
\node[vertex] (b) at  (1,1) {1};
\node[vertex] (c) at  (2,2) {2};
\node[vertex] (d) at  (3,3) {3};
\draw[edge] (c) to (b);
\draw[edge] (c) to (d);
\end{tikzpicture}
\qquad
\begin{tikzpicture}
\tikzset{vertex/.style = {shape=circle,draw,minimum size=1.5em}}
\tikzset{edge/.style = {->,> = latex'}}
\node[vertex] (b) at  (1,1) {1};
\node[vertex] (c) at  (2,2) {2};
\node[vertex] (d) at  (3,3) {3};
\draw[edge] (c) to (b);
\draw[edge] (d) to (c);
\end{tikzpicture}\vspace{.15in}

\begin{tikzpicture}
\tikzset{vertex/.style = {shape=circle,draw,minimum size=1.5em}}
\tikzset{edge/.style = {->,> = latex'}}
\node[vertex] (b) at  (0,1) {1};
\node[vertex] (c) at  (-1,2) {2};
\node[vertex] (d) at  (1,2) {3};
\draw[edge] (b) to (c);
\draw[edge] (b) to (d);
\end{tikzpicture}
\qquad
\begin{tikzpicture}
\tikzset{vertex/.style = {shape=circle,draw,minimum size=1.5em}}
\tikzset{edge/.style = {->,> = latex'}}
\node[vertex] (b) at  (0,1) {1};
\node[vertex] (c) at  (-1,2) {2};
\node[vertex] (d) at  (1,2) {3};
\draw[edge] (c) to (b);
\draw[edge] (b) to (d);
\end{tikzpicture}
\qquad
\begin{tikzpicture}
\tikzset{vertex/.style = {shape=circle,draw,minimum size=1.5em}}
\tikzset{edge/.style = {->,> = latex'}}
\node[vertex] (b) at  (0,1) {1};
\node[vertex] (c) at  (-1,2) {2};
\node[vertex] (d) at  (1,2) {3};
\draw[edge] (b) to (c);
\draw[edge] (d) to (b);
\end{tikzpicture}
\qquad
\begin{tikzpicture}
\tikzset{vertex/.style = {shape=circle,draw,minimum size=1.5em}}
\tikzset{edge/.style = {->,> = latex'}}
\node[vertex] (b) at  (0,1) {1};
\node[vertex] (c) at  (-1,2) {2};
\node[vertex] (d) at  (1,2) {3};
\draw[edge] (c) to (b);
\draw[edge] (d) to (b);
\end{tikzpicture}
\caption{The possible ways of directing the edges of the line and fork trees.}
\label{directedlineforktree}
\end{center}
\end{figure}
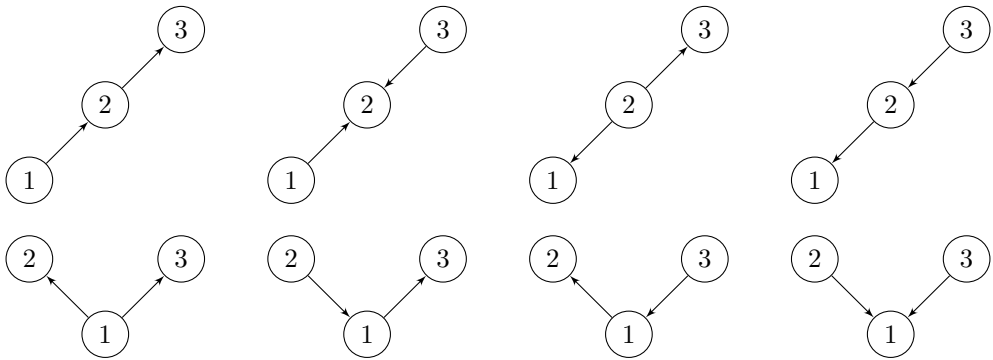

We can identify each of the $2^{n-1}$ directed trees constructed via directing the edges of a rooted tree $T$ with an $n$-tuple $k\in \{\pm 1\}^n$, where $k(1):=1$ and, letting $p$ be the parent vertex of $i$, $$k(i):=\left\{\begin{array}{cc}1 & \mbox{if the edge }   (p,i) \mbox{ is directed towards } i \\-1 & \mbox{if the edge }   (p,i) \mbox{ is directed towards } p\end{array}\right..$$ Let us denote all directed trees constructed via directing the edges of a root tree $T$ according to $k$ by $T_k$. Let $T_{k, j}$ be the subtree of $T_k$ in which $j$ is the root. For each directed tree $T_k$ we define a corresponding polytope $Q_{T_k}(t_1,...,t_n)$, which is the collection of all points $r\in \mathbb{R}_+^n$ that satisfy, for each $j\in T_k$, the set of inequalities, $$\sum_{i\in T_{k, j}} r_i \mbox{ } (\leq,\geq)_j \sum_{i\in T_{k, j}} t_i, \mbox{ } j=1,...,n,$$ where $$(\leq,\geq)_j:=\left\{\begin{array}{cc}\leq & k(j)=1 \\\geq & k(j)=-1\end{array}\right..$$

As $k(1)=1$, we have that one of the above inequalities will always be $$\sum_{i=1}^n r_i\leq 1,$$ which defines our regular simplex $\mathfrak{Sim}_n(t_1,...,t_n)$. Thus the systems of inequalities for each of our $2^{n-1}$ directed trees together give a partition of $\mathfrak{Sim}_n(t_1,...,t_n)$ into $2^{n-1}$ polytopes. For $\mathfrak{Sim}_3(t_1,t_2,t_3)$, the inequalities for the line tree are displayed in figure \ref{n3exampleline} and the inequalities for the fork tree are displayed in figure \ref{n3examplefork}.

Now take a tree $T$ with $n$ vertices. For each directed tree $T_k$, we will complete it to a particular directed graph $GBW_n(T_k)$, which we will refer to as the \Dfn{generalized broken wheel graph over $T_k$}. We construct $GBW_n(T_k)$ in the following way:
\begin{enumerate}
\item Add one more vertex, labelled $0$.
\item Add two edges from $0$ to the root vertex.
\item Add one edge from $0$ to each of the $n-1$ vertices of $T_k$.
\end{enumerate}
Let $GBW_n(T)$ denote the graph $GBW_n(T_k)$ without directed edges; $GBW_n(T)$ is the same for any $k$ and will be referred to as the \Dfn{generalized broken wheel graph over $T$}. In figure \ref{n3exampleline} we can see the graphs resulting from the line tree and in figure \ref{n3examplefork} we can see the graphs resulting from the fork tree. Once we have completed a directed tree $T_k$ to $GBW_n(T_k)$, we will assign a weight to each of its vertices. The weight $w_{T_k}(v)$ of each vertex $v$ of $GBW_n(T_k)$ will be equal to its indegree minus 1: $w_{T_k}(v):=\mbox{indeg}(v)-1.$ For instance, the weights of the $n=3$ graphs are displayed in blue above each vertex in figures \ref{n3exampleline} and \ref{n3examplefork}.

It is significant to note that $GBW_n(T)$, where $T$ is the ``line'' tree on $n$ vertices, is exactly the broken wheel graph $BW_n$; hence the name \textit{generalized} broken wheel graph. In fact, the zonotopal algebra derived from $BW_n$ is exactly the same as that which is derived from $GBW_n(T_k)$, where $T$ is a line tree and $k=(1,...,1)$.

\begin{figure}[h!]
\begin{center}
$$\begin{tabular}{*{2}{|l|}*{2}{|l|}*{2}{|l|}*{2}{|l|}}\hline \mbox{k} & (1,1,1) & (1,1,-1) \\\hline\hline \mbox{$GBW_n(T_k)$} & \begin{tikzpicture}
\tikzset{vertex/.style = {shape=circle,draw,minimum size=1.5em}}
\tikzset{edge/.style = {->,> = latex'}}
\node[vertex] (a) at  (2,0) {0};
\node[vertex] (b) at  (1,1) {1};
\node[vertex] (c) at  (2,2) {2};
\node[vertex] (d) at  (3,3) {3};
\node[above left] at (b.60) {{\color{blue} 1}};
\node[above left] at (c.60) {{\color{blue} 1}};
\node[above left] at (d.60) {{\color{blue} 1}};
\draw[edge] (b) to (c);
\draw[edge] (c) to (d);

\draw[edge] (a)  to[bend left] (b);
\draw[edge] (a)  to[bend right] (b);

\draw[edge] (a)  to[bend right] (c);
\draw[edge] (a)  to[bend right] (d);
\end{tikzpicture} & \begin{tikzpicture}
\tikzset{vertex/.style = {shape=circle,draw,minimum size=1.5em}}
\tikzset{edge/.style = {->,> = latex'}}
\node[vertex] (a) at  (2,0) {0};
\node[vertex] (b) at  (1,1) {1};
\node[vertex] (c) at  (2,2) {2};
\node[vertex] (d) at  (3,3) {3};
\node[above left] at (b.60) {{\color{blue} 1}};
\node[above left] at (c.60) {{\color{blue} 2}};
\node[above left] at (d.60) {{\color{blue} 0}};
\draw[edge] (b) to (c);
\draw[edge] (d) to (c);

\draw[edge] (a)  to[bend left] (b);
\draw[edge] (a)  to[bend right] (b);

\draw[edge] (a)  to[bend right] (c);
\draw[edge] (a)  to[bend right] (d);
\end{tikzpicture} \\\hline \mbox{$Q_{T_k}(t_1,t_2,t_3)$} & $\begin{array} {rll} r_1+r_2+r_3 & \leq & t_1+t_2+t_3 \\ r_2+r_3 & \leq & t_2+t_3 \\ r_3 & \leq & t_3 \end{array}$ & $\begin{array} {rll} r_1+r_2+r_3 & \leq & t_1+t_2+t_3 \\ r_2+r_3 & \leq & t_2+t_3 \\ r_3 & \geq & t_3 \end{array}$  \\\hline \mbox{$\Ref_{T_k}$} & $t_1t_2t_3$ & $t_1t_2^2$ \\\hline \mbox{$q_{T_k}(t)$} & $t_1t_2t_3+t_2^2t_3+t_3^3+t_1t_3^2+t_2t_3^2$ & $t_1t_2^2+t_2^3$ \\\hline \end{tabular}$$

$$\begin{tabular}{*{2}{|l|}*{2}{|l|}*{2}{|l|}*{2}{|l|}}\hline \mbox{k} & (1,-1,1) & (1,-1,-1) \\\hline\hline \mbox{$GBW_n(T_k)$} & \begin{tikzpicture}
\tikzset{vertex/.style = {shape=circle,draw,minimum size=1.5em}}
\tikzset{edge/.style = {->,> = latex'}}
\node[vertex] (a) at  (2,0) {0};
\node[vertex] (b) at  (1,1) {1};
\node[vertex] (c) at  (2,2) {2};
\node[vertex] (d) at  (3,3) {3};
\node[above left] at (b.60) {{\color{blue} 2}};
\node[above left] at (c.60) {{\color{blue} 0}};
\node[above left] at (d.60) {{\color{blue} 1}};
\draw[edge] (c) to (b);
\draw[edge] (c) to (d);

\draw[edge] (a)  to[bend left] (b);
\draw[edge] (a)  to[bend right] (b);

\draw[edge] (a)  to[bend right] (c);
\draw[edge] (a)  to[bend right] (d);
\end{tikzpicture} & \begin{tikzpicture}
\tikzset{vertex/.style = {shape=circle,draw,minimum size=1.5em}}
\tikzset{edge/.style = {->,> = latex'}}
\node[vertex] (a) at  (2,0) {0};
\node[vertex] (b) at  (1,1) {1};
\node[vertex] (c) at  (2,2) {2};
\node[vertex] (d) at  (3,3) {3};
\node[above left] at (b.60) {{\color{blue} 2}};
\node[above left] at (c.60) {{\color{blue} 1}};
\node[above left] at (d.60) {{\color{blue} 0}};
\draw[edge] (c) to (b);
\draw[edge] (d) to (c);

\draw[edge] (a)  to[bend left] (b);
\draw[edge] (a)  to[bend right] (b);

\draw[edge] (a)  to[bend right] (c);
\draw[edge] (a)  to[bend right] (d);
\end{tikzpicture} \\\hline \mbox{$Q_{T_k}(t_1,t_2,t_3)$} & $\begin{array} {rll} r_1+r_2+r_3 & \leq & t_1+t_2+t_3 \\ r_2+r_3 & \geq & t_2+t_3 \\ r_3 & \leq & t_3 \end{array}$ & $\begin{array} {rll} r_1+r_2+r_3 & \leq & t_1+t_2+t_3 \\ r_2+r_3 & \geq & t_2+t_3 \\ r_3 & \geq & t_3 \end{array}$ \\\hline \mbox{$\Ref_{T_k}$} & $t_1^2t_3$ & $t_1^2t_2$ \\\hline \mbox{$q_{T_k}(t)$} & $t_1^2t_3$ & $t_1^2t_2+t_1^3$ \\\hline \end{tabular}$$
\caption{$\mathfrak{Sim}_3$ with the line tree.}
\label{n3exampleline}
\end{center}
\end{figure}

\begin{figure}[h!]
\begin{center}
$$
\begin{tabular}{*{2}{|l|}*{2}{|l|}*{2}{|l|}*{2}{|l|}}\hline $k$ & (1,1,1) & (1,-1,1) \\\hline\hline \mbox{$GBW_n(T_k)$} & \begin{tikzpicture}
\tikzset{vertex/.style = {shape=circle,draw,minimum size=1.5em}}
\tikzset{edge/.style = {->,> = latex'}}
\node[vertex] (a) at  (0,-.3) {0};
\node[vertex] (b) at  (0,1) {1};
\node[vertex] (c) at  (-1,2) {2};
\node[vertex] (d) at  (1,2) {3};
\node[above left] at (b.60) {{\color{blue} 1}};
\node[above left] at (c.60) {{\color{blue} 1}};
\node[above left] at (d.60) {{\color{blue} 1}};
\draw[edge] (b) to (c);
\draw[edge] (b) to (d);

\draw[edge] (a)  to[bend left] (b);
\draw[edge] (a)  to[bend right] (b);

\draw[edge] (a)  to[bend left] (c);
\draw[edge] (a)  to[bend right] (d);
\end{tikzpicture} & \begin{tikzpicture}
\tikzset{vertex/.style = {shape=circle,draw,minimum size=1.5em}}
\tikzset{edge/.style = {->,> = latex'}}
\node[vertex] (a) at  (0,-.3) {0};
\node[vertex] (b) at  (0,1) {1};
\node[vertex] (c) at  (-1,2) {2};
\node[vertex] (d) at  (1,2) {3};
\node[above left] at (b.60) {{\color{blue} 2}};
\node[above left] at (c.60) {{\color{blue} 0}};
\node[above left] at (d.60) {{\color{blue} 1}};
\draw[edge] (c) to (b);
\draw[edge] (b) to (d);

\draw[edge] (a)  to[bend left] (b);
\draw[edge] (a)  to[bend right] (b);

\draw[edge] (a)  to[bend left] (c);
\draw[edge] (a)  to[bend right] (d);
\end{tikzpicture} \\\hline \mbox{$Q_{T_k}(t_1,t_2,t_3)$} & $\begin{array} {rll} r_1+r_2+r_3 & \leq & t_1+t_2+t_3 \\ r_2 & \leq & t_2 \\ r_3 & \leq & t_3 \end{array}$ & $\begin{array} {rll} r_1+r_2+r_3 & \leq & t_1+t_2+t_3 \\ r_2 & \geq & t_2 \\ r_3 & \leq & t_3 \end{array}$ \\\hline \mbox{$\Ref_{T_k}$} & $t_1t_2t_3$ & $t_1^2t_3$ \\\hline \mbox{$q_{T_k}(t)$} & $t_1t_2t_3+t_2^2t_3+t_2t_3^2$ & $t_1^2t_3+t_1t_3^2+t_3^2$ \\\hline \end{tabular}
$$

$$
\begin{tabular}{*{2}{|l|}*{2}{|l|}*{2}{|l|}*{2}{|l|}}\hline $k$ & (1,1,-1) & (1,-1,-1) \\\hline\hline \mbox{$GBW_n(T_k)$} & \begin{tikzpicture}
\tikzset{vertex/.style = {shape=circle,draw,minimum size=1.5em}}
\tikzset{edge/.style = {->,> = latex'}}
\node[vertex] (a) at  (0,-.3) {0};
\node[vertex] (b) at  (0,1) {1};
\node[vertex] (c) at  (-1,2) {2};
\node[vertex] (d) at  (1,2) {3};
\node[above left] at (b.60) {{\color{blue} 2}};
\node[above left] at (c.60) {{\color{blue} 1}};
\node[above left] at (d.60) {{\color{blue} 0}};
\draw[edge] (b) to (c);
\draw[edge] (d) to (b);

\draw[edge] (a)  to[bend left] (b);
\draw[edge] (a)  to[bend right] (b);

\draw[edge] (a)  to[bend left] (c);
\draw[edge] (a)  to[bend right] (d);
\end{tikzpicture} & \begin{tikzpicture}
\tikzset{vertex/.style = {shape=circle,draw,minimum size=1.5em}}
\tikzset{edge/.style = {->,> = latex'}}
\node[vertex] (a) at  (0,-.3) {0};
\node[vertex] (b) at  (0,1) {1};
\node[vertex] (c) at  (-1,2) {2};
\node[vertex] (d) at  (1,2) {3};
\node[above left] at (b.60) {{\color{blue} 3}};
\node[above left] at (c.60) {{\color{blue} 0}};
\node[above left] at (d.60) {{\color{blue} 0}};
\draw[edge] (c) to (b);
\draw[edge] (d) to (b);

\draw[edge] (a)  to[bend left] (b);
\draw[edge] (a)  to[bend right] (b);

\draw[edge] (a)  to[bend left] (c);
\draw[edge] (a)  to[bend right] (d);
\end{tikzpicture} \\\hline \mbox{$Q_{T_k}(t_1,t_2,t_3)$} & $\begin{array} {rll} r_1+r_2+r_3 & \leq & t_1+t_2+t_3 \\ r_2 & \leq & t_2 \\ r_3 & \geq & t_3 \end{array}$ & $\begin{array} {rll} r_1+r_2+r_3 & \leq & t_1+t_2+t_3 \\ r_2 & \geq & t_2 \\ r_3 & \geq & t_3 \end{array}$ \\\hline \mbox{$\Ref_{T_k}$} & $t_1^2t_2$ & $t_1^3$ \\\hline \mbox{$q_{T_k}(t)$} & $t_1^2t_2+t_1t_2^2+t_2^3$ & $t_1^3$ \\\hline \end{tabular}
$$
\caption{$\mathfrak{Sim}_3$ with the fork tree.}
\label{n3examplefork}
\end{center}
\end{figure}

\subsection{The Zonotopal Spaces of the Generalized Broken Wheel Graph} 

The weights of the vertices of $GBW_n(T_k)$ will guide us in constructing a polynomial $q_{T_k}(t)\in \mathbb{K}[t_1,...,t_n]$, where $\mathbb{K}$ is a field of characteristic $0$, which will turn out to be the volume of the polytope $Q_{T_k}(t_1,...,t_n)$. Each polynomial $q_{T_k}(t)$ has a distinguished monomial $$\Ref_{T_k}: t\mapsto t^{w_{T_k}}:= \prod_{i=1}^n t_i^{w_{T_k}(i)}, \mbox{ } w_{T_k}:=(w_{T_k}(1),...,w_{T_k}(n)),$$ called the \Dfn{reference monomial} of $T_k$. The polynomial $q_{T_k}(t)$ is constructed in the following way: the reference monomial $\Ref_{T_k}$ is a term of $q_{T_k}(t)$. To get the exponent vectors of the other terms of $q_{T_k}(t)$, let's think of the weight at each vertex $i$ of $T_k$ as a sandpile of $w_{T_k}(i)$ grains of sand. Each grain of sand can be moved to a sandpile at another vertex $j$ if there is an edge directed from $i$ towards $j$. 

More formally, a \Dfn{move} can be made from $i$ to $j$ if $w_{T_k}(i)>0$ and there exists an edge between $i$ and $j$ which is directed towards $j$. If a move is made from $i$ to $j$, then the weight at $i$ becomes $w_{T_k}(i)-1$ and the weight at $j$ becomes $w_{T_k}(j)+1$. We then have that $w\in \mathrm{supp}\mbox{ } q_{T_k}(t)$ if a series of moves can be made to get $w$ from $w_{T_k}$.

\begin{example}
Consider the top, leftmost graph in figure \ref{n3examplefork} with $k=(1,1,1)$. We know that $\Ref_{T_k}=t_1t_2t_3$ is a term of $q_{T_k}(t)$. Remembering that we always start at $w_{T_k}=(1,1,1)$, we can see that a move can be made from $1$ to $2$ to get $(0,2,1)$, giving us the term $t_2^2t_3$. We can also make a move from $1$ to $3$ to get $(0,1,2)$, giving us the term $t_2t_3^2$. As there are no other tuples which can be reached by a series of moves, we have that $\mathrm{supp}\mbox{ } q_{T_k}=\{(1,1,1),(0,1,2),(0,2,1)\}$ and $q_{T_k}(t)=t_1t_2t_3+t_2^2t_3+t_2t_3^2.$
\end{example}

We can construct zonotopal spaces from $GBW_n(T)$ in a similar fashion as we did for $BW_n$. For every edge $(i,j)$ of $GBW_n(T)$ we associate the vector $e_i-e_j$ if $(i,j)$ is directed towards $i$ and $e_j-e_i$ if $(i,j)$ is directed towards $j$. We take these vectors as columns of a matrix $GX_n$. From this matrix we can construct the central, internal, and external pairs of zonotopal spaces, as described in section \ref{introsonotopal}.

Let $\mathcal{P}_{n}(GX_n)$ be the space of all homogeneous polynomials of degree $n$ that lie in the $\mathcal{P}$-central space $\mathcal{P}(GX_n)$, and let $\mathcal{D}_{n}(GX_n)$ be the space of all homogeneous polynomials of degree $n$ that lie in the $\mathcal{D}$-central space $\mathcal{D}(GX_n)$. We will now show that the polynomials $q_{T_k}(t)$ form a basis for $\mathcal{D}_{n}(GX_n)$ and that the reference monomials $\Ref_{T_k}$ form a basis for $\mathcal{P}_{n}(GX_n)$.

\begin{theorem}
$\mathcal{P}_{n}(GX_n)$ is monomial and the monomials $\Ref_{T_k}$ for each $k$ together form a basis for $\mathcal{P}_{n}(GX_n)$.
\end{theorem}

\begin{proof}
Benson, Chakrabarty, and Tetali prove in Theorem 3.1 of \cite{2008arXiv0801.1114B} that the set of weights, $$\{w_{T_k} : k\in \{\pm 1\}^n, k(1)=1\},$$ is exactly the set of maximal parking functions of $GBW_n(T_k)$. It was then shown in \cite{2003math......1110P} that the set of parking functions of any graph $G$ is the support of a monomial basis of the $\mathcal{P}$-central space associated to $G$. Thus the set of reference monomials, $\{\Ref_{T_k} : k\in \{\pm 1\}^n, k(1)=1\},$ is exactly the degree $n$ basis elements of the $\mathcal{P}$-central space $\mathcal{P}(GX_n)$, which generate $\mathcal{P}_{n}(GX_n)$.
\end{proof}

\begin{theorem}
The polynomials $q_{T_k}(t)$ are contained in and form a basis for $\mathcal{D}_{n}(GX_n)$.
\end{theorem}

\begin{proof}
A polynomial is contained in $\mathcal{D}_{n}(GX_n)$ if it is homogeneous of degree $n$ and annihilated by all the operators defined by the cocircuits of $GBW_n(T_k)$. Let's consider any cocircuit $C$ of $GBW_n(T_k)$. We know that $C$ is defined by a cycle in the dual graph of $GBW_n(T_k)$; let the set $\{{v_1},...,{v_s}\}$ be the set of vertices which are dual to $C$. The operator $D_C$ defined by $C$ is the product of operators of the form $(D_x-D_y)$ where $(x,y)$ is an edge in $GBW_n(T_k)$ dual to an edge of $C$. We can see that the operator $D_{v_1}\cdots D_{v_s}$ is a factor of $D_C$, as all edges $(0,v_i), 1\leq i\leq s$, are dual to an edge of $C$.

If $D_{v_1}\cdots D_{v_s}$ does not annihilate $q_{T_k}(t)$, then there must exist a vertex $v_i$ in $\{v_1,...,v_s\}$ such that all edges from $v_i$ to a vertex in $\{v_1,...,v_s\}\backslash \{v_i\}$ flow out of $v_i$, and such that all edges from $v_i$ to a vertex in $\{v_{s+1},...,v_n\}$ flow into $v_i$. The product of all operators $D_{v_j}$ such that $v_j$ is adjacent to $v_i$ and $v_j\in \{v_{s+1},...,v_n\}$ is a factor of $D_C$ and annihilates $q_{T_k}(t)$ together with $D_{v_1}\cdots D_{v_s}$, giving us that $q_{T_k}(t)\in \mathcal{D}_{n}(GX_n)$. The polynomials $q_{T_k}(t)$ are then the unique $s$-monic polynomials, where $s$ is the support of some reference monomial, which form a basis for $\mathcal{D}_{n}(GX_n)$ by proposition \ref{r4}.
\end{proof}

\begin{theorem}
The volume of $Q_{T_k}(t_1,...,t_n)$ is $q_{T_k}(t)$.
\end{theorem}

\begin{proof}
The truncated power $Trn_X(t)$ is a function which records the normalized volume of $Q_{T_k}(t_1,...,t_n)$. As defined in \cite{de1993box}, it can specifically be identified as the function $$Trn_X(t):=vol_{n-d}(X^{-1}\{t\}\cap \mathbb{R}_+^n)dt/ |\det X|, \text{ } t\in \ran X,$$ where $\ran X$ is the range of $X$, $d$ is the dimension of $\ran X$, and $X$ is any matrix in which $0$ is an extreme point for the non-negative polytope $M_X$ whose closed support is given by $$\mathrm{supp}\mbox{ } M_X= \{Xa : 0\leq a\leq 1\}.$$

It is piecewise in the $\mathcal{D}$-central space of $GBW_n(T)$, which is spanned by the $2^{n-1}$ polynomials $q_{T_k}(t)$. Since no edge of $GBW_n(T_k)$ ever lies in the interior of the positive octant for any $k$, the volume is one polynomial piece in the positive octant.\

The positive octant has $n$ facets. At least one of these facets is a part of the boundary of the support of the truncated power. The facets on the boundary depend on the $k$ we choose. The volume polynomial $Q_{T_k}(t_1,...,t_n)$ is thus divisible by $t_i^{w_i}$ whenever $t_i=0$ is a boundary facet and $w_i+1$ edges do not lie in the $t_i=0$ facet. In our case, $i$ will be a vertex which is a sink and $w_i$ its corresponding weight.

For $n=2$, there is one tree $T$ with two possible orientations: $k_1=(1,1)$ and $k_2=(1,-1)$. We then know that the polynomials $q_{T_{k_1}}=t_1t_2+t_2^2$ and $q_{T_{k_2}}=t_1^2$ form a basis for $\mathcal{D}(GX_2)$; so the volumes of $Q_{T_{k_1}}(t_1,...,t_n)$ and $Q_{T_{k_2}}(t_1,...,t_n)$ must be linear combinations of $q_{T_{k_1}}=t_1t_2+t^2$ and $q_{T_{k_2}}=t_1^2$, respectively. As these polynomials are divisible by $t_2$ and $t_1^2$, respectively, we can see from our observations about the truncated power that the volume of $Q_{T_{k_1}}(t_1,...,t_n)$ must be $q_{T_{k_1}}=t_1t_2+t_2^2$ and the volume of $Q_{T_{k_2}}(t_1,...,t_n)$ must be $q_{T_{k_2}}=t_1^2$.

Let's assume that the volume of $Q_{T_k}(t_1,...,t_n)$ is $q_{T_k}(t)$ $n>2$ for any $k$, and consider any tree $T$ with $n+1$ vertices, a $k$, and $GBW_{n+1}(T_k)$. We would like to find the volume of $Q_{T_k}(t_1,...,t_n,t_{n+1})$. We can pick a leaf $l$ of $GBW_{n+1}(T_k)$ with parent $p$, and consider the polytope $Q_{T_{k}}(t_1,...,t_{l-1},t_{l+1},...,t_{n+1})$ corresponding to the directed graph resulting from removing the edge between $l$ and $p$ and the edge between $l$ and $0$. We have two cases to consider: the case where the edge connecting $l$ and $p$ is oriented from $p$ to $l$, and the case where the edge connecting $l$ and $p$ is oriented from $l$ to $p$. For each case respectively, we have that:
\begin{enumerate}
\item If the edge connecting $l$ and $p$ is oriented from $p$ to $l$, then $$(D_l-D_p)vol(Q_{T_k}(t_1,...,t_n,t_{n+1}))=vol(Q_{T_{k}}(t_1,...,t_{l-1},t_{l+1},...,t_{n+1})).$$
\item If the edge connecting $l$ and $p$ is oriented from $l$ to $p$, then $$(D_p-D_l)vol(Q_{T_k}(t_1,...,t_n,t_{n+1}))=vol(Q_{T_{k}}(t_1,...,t_{l-1},t_{l+1},...,t_{n+1})).$$
\end{enumerate}

Let us first begin with the case where the edge connecting $l$ and $p$ is oriented from $p$ to $l$, as it is the quickest. In this case, as $vol(Q_{T_k}(t_1,...,t_n))$ has all positive coefficients, we know that $$vol(Q_{T_k}(t_1,...,t_n,t_{n+1}))=t_l\cdot vol(Q_{T_{k}}(t_1,...,t_{l-1},t_{l+1},...,t_{n+1})) + t_lD_p\cdot vol(Q_{T_k}(t_1,...,t_n,t_{n+1})).$$ Graphically this is the same as adding $1$ to the weight of $l$, and then adding a monomial for each time you make a move from $p$ to $l$. Or in other words, $vol(Q_{T_k}(t_1,...,t_n,t_{n+1}))=q_{T_k}(t)$.

The second case, where the edge connecting $l$ and $p$ is oriented from $l$ to $p$, is a bit more subtle. This is because we need to consider whether or not $p$ is a sink. If $p$ is a sink, then there is a $t_p$ in every monomial of $q_{T_k}(t)$ and never a $t_l$. And so we can see that $vol(Q_{T_k}(t_1,...,t_n,t_{n+1}))=t_pQ_{T_{k}}(t_1,...,t_{l-1},t_{l+1},...,t_{n+1})$, which is the same as adding $1$ to the weight of $p$, showing us that $vol(Q_{T_k}(t_1,...,t_n,t_{n+1}))=q_{T_k}(t)$.

When $p$ is not a sink, we have to be careful because it is difficult to recover what we have lost after applying $(D_p-D_l)$ to $vol(Q_{T_k}(t_1,...,t_n))$ from $vol(Q_{T_{k}}(t_1,...,t_{l-1},t_{l+1},...,t_{n+1}))$, as we can no longer keep track of what moves out of $p$. So let's assume that $vol(Q_{T_k}(t_1,...,t_n))= q_{T_k}(t)+ q_{T_{k'}}(t)$, where $k\neq k'$ and $q_{T_{k'}}(t)$ is divisible by a sink of $T_k$ raised to the power of its weight. When applying $(D_p-D_l)$ to $vol(Q_{T_k}(t_1,...,t_n))$, we can see that $(D_p-D_l)q_{T_k}(t)=D_pq_{T_k}(t)=vol(Q_{T_{k}}(t_1,...,t_{l-1},t_{l+1},...,t_{n+1}))$, as there are no terms with $t_l$ in $q_{T_k}(t)$. 

We must then have that $(D_p-D_l)q_{T_{k'}}(t)=0$. This means that either there are is no $t_p$ or $t_l$ as a factor of any term in $q_{T_{k'}}(t)$, which is not possible as the edge connecting $p$ and $l$ must be oriented towards either $p$ or $l$, or there exists a pair of terms, $t_p\alpha$ and $t_l\alpha$, of $q_{T_{k'}}(t)$, where $\alpha$ is a monomial in $\mathbb{K}[t_1,...,t_n]$. This can only be the case if the edge connecting $p$ and $l$ is oriented towards $l$ in $T_{k'}$, as that is the only way for there to even exist a term with a factor of $t_l$ to begin with. This means, however, that $(D_p-D_l)q_{T_{k'}}(t)=-vol(Q_{T_{k'}}(t_1,...,t_{l-1},t_{l+1},...,t_{n+1}))$ by the first case we considered in this proof. As $-vol(Q_{T_{k'}}(t_1,...,t_{l-1},t_{l+1},...,t_{n+1}))$ is non-zero, this contradicts the fact that $(D_p-D_l)vol(Q_{T_k}(t_1,...,t_n,t_{n+1}))=vol(Q_{T_{k}}(t_1,...,t_{l-1},t_{l+1},...,t_{n+1}))$. We must then have that $vol(Q_{T_k}(t_1,...,t_n,t_{n+1}))=q_{T_k}(t)$, as desired.
\end{proof}

With these results we can see that the zonotopal algebra derived from a given rooted tree $T$ completely describes a polyhedral subdivision of $\mathfrak{Sim}_n(t_1,...,t_n)$. Given how the zonotopal spaces in our study seem to capture the volumes of the polytopes and the polytopes appearing in their various subdivisions, it seems fair to suggest that the volumes of polytopes in general could be studied via their corresponding zonotopal spaces. Given a polytope, one would need to ask what the appropriate graphical matroid would be to derive the zonotopal spaces which capture its volume, and then analyze which polyhedral  subdivisions come out of these spaces. This method could be a new and interesting approach towards studying volumes of polytopes.

\nocite{*}
\bibliographystyle{alpha}
\bibliography{Brodsky}

\end{document}